\documentclass[aps,pre,reprint,amsmath,amssymb,superscriptaddress,floatfix]{revtex4-2}

\usepackage[T1]{fontenc}
\usepackage{lmodern}
\usepackage{amsthm}
\usepackage{graphicx}
\usepackage{booktabs}
\usepackage{multirow}
\usepackage{xcolor}
\usepackage{tikz}
\usetikzlibrary{arrows.meta,positioning,calc}
\usepackage[hidelinks]{hyperref}

\newtheorem{proposition}{Proposition}
\newtheorem{corollary}{Corollary}
\theoremstyle{remark}
\newtheorem{remark}{Remark}
\theoremstyle{definition}

\newcommand{\Tst}{T_{\mathrm{st}}}
\newcommand{\Tpop}{T_{\mathrm{pop}}}
\newcommand{\Tcol}{T_{\mathrm{col}}}
\newcommand{\Tgeo}{T_{\mathrm{geo}}}
\newcommand{\TG}{T_{\Gamma}}
\newcommand{\Tdiv}{T_{\mathrm{div}}}
\newcommand{\Z}{\mathbb{Z}}
\newcommand{\prelim}[1]{#1}
\newcommand{\prelimL}[1]{#1}

\begin{document}

\title{Symmetry-resolved periods in a cyclic-dominance automaton: separating intrinsic from lattice-induced periodicity}

\author{Francisco J. Mu\~noz}
\email{fcomunoz@alumnos.upm.es}
\affiliation{Departamento de Matem\'atica Aplicada, Universidad Polit\'ecnica de Madrid, 28040 Madrid, Spain}
\author{Juan Carlos Nu\~no}
\affiliation{Departamento de Matem\'atica Aplicada, Universidad Polit\'ecnica de Madrid, 28040 Madrid, Spain}

\begin{abstract}
Periodic attractors of spatially extended systems possess not one period but a
hierarchy of periods, one for each level of observation and symmetry quotient,
and this hierarchy separates how a pattern evolves from how it fits the finite
domain. For a synchronous threshold model of
cyclic dominance on an $L\times L$ torus, equivalent to the three-color cyclic
cellular automaton of Fisch, Gravner and Griffeath, we measure for every
attractor a \emph{tuple} of periods: that of the populations ($\Tpop$), of the
full configuration ($\Tst$), and of the configuration up to color rotations
($\Tcol$), lattice symmetries ($\Tgeo$) and the full space--color group ($\TG$).
Equivariance implies exact divisibility relations and the factorization
$\Tst=\TG\,m$, where $m$ counts the applications of the symmetry $h$ that maps
the configuration onto itself after $\TG$ steps: $\TG$ and the type of $h$
belong to the pattern, $m$ to the lattice, and the attractors are relative
periodic orbits whose full period emerges from the interaction between the
intrinsic dynamics and the lattice symmetries. A family of attractors built around a
pair of counter-rotating spirals, with $\Tpop=24$ and $\TG=8$, persists from
$L=9$ to $L=40$ while $\Tst$ takes values from 24 to 960, given in every case
by the order of $h$; on a rectangular $9\times12$ torus it takes the value 48,
absent from all square tori. With more than $4.5\times10^5$ runs at each of
$L=8$ and $L=9$ we estimate the basins of attraction: fixed points take
95--97\% of the configuration space, and no single period is enough to tell
the oscillating attractors apart. The same analysis locates local diversity: it is
balanced whenever the cycle has a color symmetry and, on large lattices, it is
concentrated at the spiral cores.
\end{abstract}

\maketitle

\section{Introduction}
\label{sec:intro}

Cyclic dominance, in which three or more species or opinions beat each other in
a non-transitive way, is one of the simplest mechanisms able to sustain
diversity in spatially extended populations
\cite{Kerr2002,Reichenbach2007,Szolnoki2014,SzaboFath2007}. Lattice versions of
the rock--paper--scissors game have been studied for decades, both in their
stochastic form \cite{Tainaka1988,Frean2001,Reichenbach2007} and as deterministic
cellular automata \cite{Greenberg1978,Bramson1989,Fisch1991a,Fisch1991b}. In the
latter, the dynamics is a map on a finite set of configurations, so that every
trajectory eventually falls onto a periodic orbit (a fixed point being a cycle of
period one) \cite{Wolfram1983,Martin1984}. Threshold rules in which an agent
changes state when a given fraction of its neighbors is different have, in
turn, a long tradition in the social sciences \cite{Schelling1971,Granovetter1978},
and for symmetric (non-cyclic) threshold networks it is known that synchronous
dynamics can only produce fixed points or 2-cycles \cite{Goles1980}. The cyclic
orientation of the interaction breaks this bound and allows long cycles.

Periodicity has been a central theme in the study of cyclic cellular
automata. For the cyclic cellular automaton (CCA) on $\Z^2$, Fisch, Gravner and
Griffeath proved local periodicity: in the spiral-forming regimes every site is
eventually periodic \cite{Fisch1991b,Fisch1991a}, and the asymptotic behavior of excitable
automata has been characterized further in \cite{Durrett1991,Durrett1993}. On
finite graphs, Matamala and Goles showed that cyclic automata networks can have
non-polynomial periods, with bounds that depend on the structure of the graph
\cite{Matamala1997}, and Matamala and Moreno proved that, for finite
configurations on $\Z^2$, the period of every site divides the least common
multiple of the lengths of a finite family of cycles \cite{Matamala2004}. For
the Greenberg--Hastings model on finite graphs, whether the dynamics dies out or
becomes periodic, and the asymptotic firing rate, are determined by the phase
differences (winding numbers) around the cycles of the graph
\cite{Doba2018,Cai2018}. More recent numerical work has measured the duration
of the transient phases of the CCA on large tori \cite{Tomlinson2019}. These
studies concern the period of individual sites, bounds on the global period, or
the convergence to periodicity; to our knowledge, none of them compares the
periods obtained at different levels of description of the same attractor
(aggregated populations, full configuration, configuration up to color
relabelling, configuration up to lattice symmetries), nor uses the
relations between them to classify attractors, which is the purpose of this
work.

A central observation motivating this work is that the notion of period is not
unique in spatially extended cyclic systems: the same attractor may show
different periodicities depending on whether one observes the full state, the
pattern of domains irrespective of colors, the pattern up to its position and
orientation, or the population counts. How should a periodic attractor be
characterized, then? The usual choice is a single
number: either the period of the aggregated densities, which is what one
observes in population data, or the period of the complete microscopic
configuration. The two need not coincide, and neither alone tells us
\emph{how} the configuration comes back to itself. A pattern that travels
rigidly across the torus, for instance, has a long microscopic period but
completely flat densities; conversely, a pattern that returns to itself
with the colors permuted shows densities that oscillate with three times the
period of its ``shape''.

In this work we measure, for every attractor, a \emph{tuple of periods}
obtained by comparing configurations at different levels of coarse-graining:
the populations, the full configuration, the configuration without color labels
and the configuration up to a geometric transformation of the lattice. We show
that these periods are linked by exact divisibility relations derived from the
symmetry group of the dynamics, in the spirit of the $H/K$ theory of
spatio-temporal symmetries of periodic orbits \cite{Golubitsky2002}, and that
their ratios identify the space--color transformation that maps the attractor
onto itself after a fraction of its period. Since initial conditions are drawn
uniformly, the frequency of each tuple in a census of random runs also
estimates the basins of attraction.

The central idea of the paper is that the symmetry group separates the
properties of a pattern from those imposed by the domain. The period up to the
full group, $\TG$, and the type of the generating symmetry $h$ belong to the
pattern; the state period is $\Tst=\TG\,m$, where $m$, the number of
applications of $h$ needed to return, depends on the size and shape of the torus.
A family of spiral attractors with $\Tpop=24$ and $\TG=8$ illustrates this: it
appears at $L=9$ and persists up to $L=40$ with the same $\Tpop$ and $\TG$, while
its state period takes values between 24 and 960, and on a rectangular torus it
takes a value found on no square one (Secs.~\ref{sec:t24} and \ref{sec:tests}).
The same analysis bears on a classical question about cyclic dominance, namely
where local diversity (the coexistence of the three states within a
neighborhood) survives \cite{Kerr2002,Reichenbach2007}: the diversity counts are invariant under the full group, a color
symmetry forces balanced coexistence, and the spiral defects guarantee local
diversity at their cores (Sec.~\ref{sec:diversity}).

We study in detail $L=8$ and $L=9$, small enough to sample the basins with high
precision and to test every element of the symmetry group on every attractor,
and differing in one relevant respect: 9 is a multiple of three, which allows
symmetries that combine a translation with a color rotation. Larger square tori,
up to $L=40$, and a rectangular torus show how the picture changes with size.
The main contributions are: (i) a hierarchy of periods defined by quotients of
the configuration space by subgroups of the symmetry group; (ii) exact
divisibility relations among them and the factorization $\Tst=\TG\,m$, which
identifies the attractors as relative periodic orbits and separates intrinsic
from lattice-induced periodicity; (iii) a family of spiral attractors with
$\Tpop=24$ that demonstrates this separation over 32 square tori and a
rectangular one; (iv) a link between the symmetries of an attractor and the
local diversity it sustains; and (v) a census of the attractors and their basins
based on more than $10^7$ runs.
Section~\ref{sec:model} defines the model, Sec.~\ref{sec:periods} the tuple
of periods and its relations, Sec.~\ref{sec:methods} the simulations,
Sec.~\ref{sec:results} presents the results and Sec.~\ref{sec:discussion}
discusses them.

\section{Model}
\label{sec:model}

We consider the square lattice $\Lambda=\Z_L\times\Z_L$ with periodic boundary
conditions (a torus) and $N=L^2$ sites. Each site $i$ carries a state
$s_t(i)\in\{\mathrm{X},\mathrm{Y},\mathrm{Z}\}$, which we identify with
$\{0,1,2\}$. The neighborhood $\mathcal{N}_r(i)$ of a site is its Moore
neighborhood of radius $r$ (all sites at Chebyshev distance $1\le d\le r$),
with $k_r=|\mathcal{N}_r(i)|=(2r+1)^2-1$ neighbors when $2r+1\le L$.
The interaction is cyclic: Y dominates X, Z dominates Y and X dominates Z. We
call the dominating state of $s$ its \emph{predator}, $s^{+}=s+1 \pmod 3$.

All sites are updated synchronously according to the deterministic rule
\begin{equation}
s_{t+1}(i)=
\begin{cases}
s_t(i)+1 \pmod 3, & \text{if } \phi_t(i)\ge\alpha,\\[1mm]
s_t(i), & \text{otherwise},
\end{cases}
\label{eq:rule}
\end{equation}
where
\begin{equation}
\phi_t(i)=\frac{1}{k_r}\,\bigl|\{j\in\mathcal{N}_r(i): s_t(j)=s_t(i)+1\}\bigr|
\end{equation}
is the fraction of neighbors of $i$ in the predator state and $\alpha\in[0,1]$ is the conversion threshold. A site therefore changes
only to its predator and only when the predator is sufficiently represented in
its neighborhood. For $\alpha=0$ every site changes at every step and the
whole lattice rotates its colors, while for $\alpha=1$ a site changes only
when it is completely surrounded by its predator. Since only the integer
$\lceil\alpha k_r\rceil$ matters, the rule has $k_r+1$ distinct regimes; for
$r=1$ these are $\alpha=m/8$, $m=0,\dots,8$, which are the values used in the
$\alpha$ sweeps. We write $F$ for the global map
$s_t\mapsto s_{t+1}$.

In the language of excitable media, Eq.~\eqref{eq:rule} is the
three-color cyclic cellular automaton with box neighborhood of range $r$ and
threshold $\theta=\lceil\alpha k_r\rceil$ studied by Fisch, Gravner and
Griffeath \cite{Fisch1991a,Fisch1991b}, with $\kappa=3$ colors, in which a
site advances when at least $\theta$ sites within range $\rho=r$ have the
successor color. In their notation the reference case $\alpha=1/3$, $r=1$ is
the rule $(\rho,\theta,\kappa)=(1\mathrm{B},3,3)$. Two of their thresholds are
directly relevant here. The \emph{band test} shows that a straight band of the
successor color propagates when $\theta\le(2\rho+1)\rho$, i.e.\ $\theta\le3$
for $\rho=1$; and above the bootstrap threshold, which for $\rho=1$ lies between
$\theta=3$ and $\theta=4$, random initial conditions fixate. Our reference
threshold is therefore the largest one that still supports traveling bands,
and we expect fixation for all $\alpha>3/8$ ($\theta\ge4$) when $r=1$; note also
that $\alpha=1/3$ and $\alpha=3/8$ define the same rule for $r=1$. Those works
focus on large lattices, where for $(1\mathrm{B},3,3)$ they report locally
periodic equilibria of spirals, and on the transitions between debris, spirals,
turbulence and fixation; here we focus on small tori, where the attractors
reached from random initial conditions can be sampled exhaustively and their
periods measured exactly.

Initial conditions are drawn uniformly at random with prescribed numbers of
sites in each state, $n_\mathrm{X}(0)=n_\mathrm{Y}(0)=P_x$ and
$n_\mathrm{Z}(0)=N-2P_x$; the default is $P_x=\mathrm{round}(N/3)$.

\section{A tuple of periods}
\label{sec:periods}

\subsection{Symmetries of the dynamics}

Two groups act on configurations. The \emph{geometric} group
$G=\Z_L^2\rtimes D_4$ contains the $L^2$ translations of the torus and the eight
rotations and reflections of the square; since the Moore neighborhood is
invariant under all of them, $F(gs)=gF(s)$ for all $g\in G$. The \emph{color}
group $C_3=\{e,\sigma,\sigma^2\}$, with $(\sigma s)(i)=s(i)+1 \pmod 3$, also
commutes with $F$, because the rule~\eqref{eq:rule} treats the three states in
the same cyclic way. Color transpositions (e.g.\ exchanging X and Y) do
\emph{not} commute with $F$: they reverse the orientation of the cycle. The full
symmetry group of the model is therefore $\Gamma=G\times C_3$.

\subsection{Definitions}

Let $s_0,s_1,\dots,s_{\Tst-1}$ be a periodic orbit of $F$ with (minimal) period
$\Tst$, which we call the \emph{state period}. For any subgroup $H\le\Gamma$
we define the period of the orbit \emph{up to $H$},
\begin{equation}
T_H=\min\{k\ge1:\ \exists\,h\in H \text{ such that } s_{t+k}=h\,s_t\}.
\label{eq:TH}
\end{equation}
Because $F$ commutes with $H$, the definition does not depend on $t$. We use
\begin{itemize}
\item $\Tst=T_{\{e\}}$, the period of the full configuration;
\item $\Tcol=T_{C_3}$, the period of the configuration ``without color''
(the \emph{uncolor} period): the first return of the pattern of domains,
irrespective of which state occupies each domain;
\item $\Tgeo=T_{G}$, the \emph{geometric} period: the first time at which the
configuration is a translated, rotated or reflected copy of itself;
\item $\TG=T_{\Gamma}$, the period up to any combination of geometric and
color transformations.
\end{itemize}
Since $F$ commutes with every element of $\Gamma$, it induces a well-defined map
on the quotient space $\{\mathrm{X},\mathrm{Y},\mathrm{Z}\}^N/\Gamma$, whose
points are the orbits $\Gamma s$, that is, configurations considered up to
translations, rotations, reflections and cyclic exchange of the colors. $\TG$ is
the period of the orbit of $\Gamma s_t$ under this reduced dynamics: it measures
how long the pattern takes to recover its shape, irrespective of where it sits on
the torus and of which species occupies each domain.
In addition we consider scalar observables. The \emph{population period}
$\Tpop$ is the least common multiple of the periods of $n_\mathrm{X}(t)$,
$n_\mathrm{Y}(t)$ and $n_\mathrm{Z}(t)$, the number of sites in each state.
At the opposite, local, end the \emph{site period} $T_i$ is the period of the
state series $s_t(i)$ of a single site $i$. Since a configuration returns exactly
when every site does, $\mathrm{lcm}_i\,T_i=\Tst$. Because a site can only move
along $\mathrm{X}\to\mathrm{Y}\to\mathrm{Z}\to\mathrm{X}$, over one state
period it completes an integer number of turns $\nu_i$, and $3\nu_i/\Tst$ is its
mean rotation frequency (color changes per step), the local quantity used in the
theory of cyclic cellular automata \cite{Fisch1991b}. If $s_{t+\TG}=h\,s_t$ with
$h=\sigma^c g$, $g\in G$, then $s_{t+\TG}(i)=\sigma^c s_t(g^{-1}i)$, so sites
related by the geometric part of $h$ have the same period.
The diversity counts $n_m(t)$, $m=1,2,3$, the number of sites whose closed
neighborhood contains exactly $m$ distinct states, define a diversity period
$\Tdiv$.

\subsection{Divisibility relations}

\begin{proposition}\label{prop:subgroup}
For every subgroup $H\le\Gamma$, the set
$K_H=\{k\in\Z:\exists\,h\in H,\ s_{t+k}=hs_t\}$ is a subgroup of $\Z$ that
contains $\Tst$. Consequently $K_H=T_H\Z$ and $T_H$ divides $\Tst$.
Moreover, if $H_1\le H_2$ then $T_{H_2}$ divides $T_{H_1}$.
\end{proposition}

\begin{proof}
Clearly $\Tst\in K_H$ (take $h=e$). If $s_{t+k_1}=h_1s_t$ and
$s_{t+k_2}=h_2s_t$, then, using $F^{k}h=hF^{k}$,
$s_{t+k_1+k_2}=F^{k_1}(h_2s_t)=h_2s_{t+k_1}=h_2h_1s_t$, so $K_H$ is closed under
addition. Since the orbit is periodic, $-k\equiv \Tst-k$ modulo $\Tst$, and
$K_H$ is closed under negation. A subgroup of $\Z$ is of the form $T\Z$ with
$T$ its smallest positive element. Finally $H_1\le H_2$ implies
$K_{H_1}\subseteq K_{H_2}$.
\end{proof}

\begin{corollary}\label{cor:lattice}
For every periodic attractor:
\begin{enumerate}
\item[(a)] $\TG \mid \Tgeo \mid \Tst$ and $\TG \mid \Tcol \mid \Tst$;
\item[(b)] $\Tst/\Tcol\in\{1,3\}$;
\item[(c)] $\Tpop \mid \Tgeo$ and $\Tpop\mid 3\,\Tcol$; the period of the
\emph{sorted} populations divides $\TG$;
\item[(d)] $\Tdiv\mid\TG$.
\end{enumerate}
\end{corollary}

\begin{proof}
(a) follows from Proposition~\ref{prop:subgroup} with $\{e\}\le G\le\Gamma$ and
$\{e\}\le C_3\le\Gamma$. (b) If $s_{t+\Tcol}=\sigma^a s_t$ then
$s_{t+3\Tcol}=\sigma^{3a}s_t=s_t$, so $\Tst\mid 3\Tcol$; combined with
$\Tcol\mid\Tst$ the ratio is $1$ or $3$. (c) Populations are invariant under
$G$, so $s_{t+\Tgeo}=gs_t$ implies $n_c(t+\Tgeo)=n_c(t)$; under $\sigma$ the
populations are cyclically permuted, so they return after at most three
applications; the multiset of populations is invariant under the whole of
$\Gamma$. (d) The number of distinct states in a neighborhood is invariant
under $G$ and under color permutations.
\end{proof}

The relations are summarized in Fig.~\ref{fig:lattice}. Their content is that
the ratios between the periods are not arbitrary: $\Tst/\Tgeo$ is the order of
the geometric transformation $g$ that maps the attractor onto itself after
$\Tgeo$ steps (e.g.\ $L/\gcd(L,v)$ for a translation by $v$ sites along an
axis), $\Tst/\Tcol$ tells whether the cycle is a color-rotating one, and
$\Tgeo/\Tpop$ measures how many geometric ``shapes'' share identical
population vectors. Each period therefore answers a different question:
\begin{itemize}
\item $\Tpop$: what would be seen in aggregated (census) data;
\item $\Tst$: when does the system return exactly to the same microstate;
\item $\Tcol$: when does the pattern of domains return, regardless of who
occupies them (does the cycle consist of three copies of one pattern with
the colors rotated?);
\item $\Tgeo$: when does the same pattern return somewhere else or in another
orientation (is the cycle a moving or rotating object?);
\item $\TG$: the length of the elementary ``motif'' of the cycle once all
symmetries are factored out.
\end{itemize}

The last relation can be made precise. It shows how the state period splits into
a part that belongs to the pattern and a part that depends on the lattice.

\begin{proposition}\label{prop:factor}
Let $h\in\Gamma$ be such that $s_{t+\TG}=h\,s_t$, and let
$m=\min\{j\ge1:\ h^j s_t=s_t\}$. Then $s_{t+k\TG}=h^k s_t$ for all $k$, and
$\Tst=\TG\,m$. The number $m$ divides the order of $h$, and equals it when
$s_t$ has no nontrivial symmetry in $\Gamma$. For $h=\sigma^c\circ\tau_v$, a color
rotation composed with a translation by $v=(v_1,v_2)$,
\begin{equation}
\mathrm{ord}(h)=\mathrm{lcm}\!\left(\frac{3}{\gcd(3,c)},\ \frac{L}{\gcd(L,v_1,v_2)}\right).
\label{eq:order}
\end{equation}
\end{proposition}

\begin{proof}
By induction, $s_{t+(k+1)\TG}=F^{\TG}(h^k s_t)=h^kF^{\TG}s_t=h^{k+1}s_t$. By
Proposition~\ref{prop:subgroup}, $\Tst$ is a multiple of $\TG$, say $\Tst=\TG\,j$,
and $s_{t+\TG j}=s_t$ holds if and only if $h^js_t=s_t$, so the smallest such $j$
is $m$. Since $h^{\mathrm{ord}(h)}=e$, $m$ divides $\mathrm{ord}(h)$, with equality
when the stabilizer of $s_t$ is trivial. Finally, $\sigma$ commutes with the
translations, so the order of $\sigma^c\circ\tau_v$ is the least common multiple
of the order of $\sigma^c$ and of the order of $\tau_v$ in $\Z_L^2$.
\end{proof}

The factorization $\Tst=\TG\,m$ separates two kinds of information. $\TG$ and
the \emph{type} of $h$ (whether it rotates the colors, the length and
direction of its translation, whether it rotates or reflects the lattice) can be
read off from a single stretch of $\TG$ steps of the pattern, and do not depend
on how large the domain is. The factor $m$, instead, depends on $L$ through the
order of the translation in $\Z_L^2$, Eq.~\eqref{eq:order}. A pattern that
reappears with the same $\TG$ and the same kind of $h$ on lattices of different
sizes will therefore keep $\TG$ (and, as we shall see, $\Tpop$), while its state
period changes with $L$ in a way predicted by Eq.~\eqref{eq:order}. In this
sense $\TG$, the period of the reduced dynamics on $\{\mathrm{X},\mathrm{Y},\mathrm{Z}\}^N/\Gamma$,
is an \emph{intrinsic} period and $\Tst$ a \emph{lattice-induced}
one. The qualification is that $\Gamma$ itself depends on the lattice (for
instance, symmetries of order three that combine a translation with a color
rotation exist only when $3\mid L$, and a rectangular torus has no rotations by
$90^\circ$), and that a pattern may not fit in a lattice at all; $\TG$ is
intrinsic in the sense that, whenever the same pattern is present, it does not
change with the size or shape of the domain.

\subsection{Relative periodic orbits}\label{sec:rpo}

In the language of equivariant dynamics, an orbit with $s_{t+\TG}=h\,s_t$ is a
\emph{relative periodic orbit} of $F$ with drift $h$
\cite{Field1980,Krupa1990,Golubitsky2002}: it is periodic, with period $\TG$, in
the quotient space, and in the full space it closes only after the drift has been
applied $m$ times. In systems with continuous symmetries relative periodic orbits
are the generic form of periodic behavior; traveling waves and rotating,
drifting or meandering spiral waves are relative equilibria and relative
periodic orbits of the Euclidean group \cite{Barkley1994,Ashwin1997}. On a finite
lattice the group is discrete, so every drift eventually closes and the orbit is
also periodic in the ordinary sense, but its period $\Tst=\TG\,m$ records the order
of the drift in $\Gamma$ rather than the dynamics of the pattern. This is the
origin of the apparently anomalous time scales of the model, such as state
periods of several hundred steps for a pattern that repeats itself every 8 steps.
The classes of Sec.~\ref{sec:classes} are the types of drift: traveling waves
are relative equilibria ($\TG=1$, $h$ a translation); color-rotating cycles have
$h=\sigma^c$; drifting long cycles combine a translation, a reflection or a
rotation with a possible color rotation; and the asymmetric cycles have trivial
drift, $h=\mathrm{id}$, so that their quotient and full periods coincide. The
clearest example is the $\Tpop=24$ family of Sec.~\ref{sec:t24}, for which
$F^{8}s=\sigma\tau_v s$: a color rotation combined with a translation by one
lattice step, whose order on the torus sets the state period.

\begin{figure}[t]
\centering
\resizebox{\columnwidth}{!}{%
\begin{tikzpicture}[node distance=9mm and 7mm, every node/.style={font=\small},
  lev/.style={draw,rounded corners=2pt,inner sep=3pt,align=center,minimum width=27mm,minimum height=11mm},
  obs/.style={draw,dashed,rounded corners=2pt,inner sep=3pt,align=center,minimum width=24mm,minimum height=11mm},
  >={Stealth[length=2mm]}]
  \node[lev] (st) {configurations\\ $\{X,Y,Z\}^N$\\ $\Tst=72$};
  \node[lev,below left=of st,xshift=6mm] (col) {mod colors $C_3$\\ $\Tcol=72$};
  \node[lev,below right=of st,xshift=-6mm] (geo) {mod lattice $G$\\ $\Tgeo=24$};
  \node[lev,below right=of col,xshift=-6mm] (G) {mod $\Gamma=G\times C_3$\\ $\TG=8$};
  \node[obs,right=12mm of geo] (pop) {populations\\ $\Tpop=24$};
  \node[obs,right=12mm of G] (div) {diversity counts,\\ sorted populations};
  \draw[->] (st) -- node[left,font=\scriptsize]{forget colors\ } (col);
  \draw[->] (st) -- node[right,font=\scriptsize]{\ forget position} (geo);
  \draw[->] (col) -- (G);
  \draw[->] (geo) -- (G);
  \draw[->,dashed] (geo) -- (pop);
  \draw[->,dashed] (G) -- (div);
\end{tikzpicture}}
\caption{The hierarchy of periods of a periodic attractor. Each box is a level of
description, obtained by taking the quotient of the configuration space by a
subgroup of the symmetry group $\Gamma$; dashed boxes are scalar observables
invariant under $G$ (populations) or under $\Gamma$ (diversity counts, sorted
populations), whose periods divide $\Tgeo$ and $\TG$, respectively; moreover
$\Tst/\Tcol\in\{1,3\}$ and $\Tpop\mid3\Tcol$ (Corollary~\ref{cor:lattice}). The
values are those of the long cycle $(24,72,72,24,8)$ at $L=9$, for which
$\Tst=\TG\,m$ with $m=9$, the order of the drift $h=\sigma\tau_{(2,0)}$.}
\label{fig:lattice}
\end{figure}
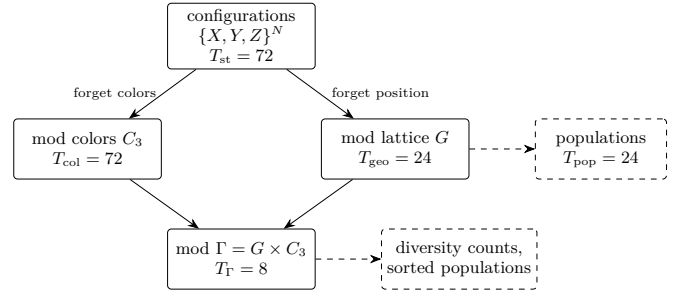

\begin{remark}[Implementation of the uncolor and geometric periods]
In the simulation code the uncolor configuration is obtained by relabelling
the states in order of first appearance along the site index. This is a
canonical form for the action of the full permutation group $S_3$ of colors,
not only of $C_3$; since transpositions are not symmetries of the dynamics the
two definitions could in principle differ, but in all the attractors we
analyzed they coincide (Sec.~\ref{sec:results}). On the other hand, applying a
lattice transformation to the relabelled configuration and comparing it with
the relabelled final state would \emph{not} test equivalence under $\Gamma$,
because the first-appearance labelling is not covariant under translations or
rotations. The simulation code therefore computes $\TG$ by testing all
$3\times 8\times L^2$ elements of $\Gamma$ directly on the colored
configurations.
\end{remark}

\section{Simulations and measurement}
\label{sec:methods}

For each parameter set we ran independent realizations from random initial
conditions, iterating Eq.~\eqref{eq:rule} until the attractor was detected.
Periods are detected on the fly: for a time series $x_0,\dots,x_n$ we
look for the smallest $p$ such that the tail $x_{\tau},\dots,x_n$ is
$p$-periodic with $n-\tau\ge 2p$, i.e.\ at least two full repetitions are
observed; an incremental algorithm keeps, for each lag $p$, the length of the
current run of coincidences $x_k=x_{k-p}$, which gives the same result as an
exhaustive search at a fraction of the cost. The run stops when the
configuration, the populations and the diversity counts have all been
found periodic (requiring at least ten confirmations), or when the lattice
becomes uniform or frozen. The transient length $\tau$ and $\Tst$ are
recorded, and $\Tgeo$ is computed afterwards by comparing the final
configuration with the $\Tst$ previous ones under all elements of $G$.

To check the recorded values independently, a sample of runs was re-simulated
from its initial condition, and the distinct final configurations for one full
period, with an independent implementation (Python/NumPy); all quantities of
Sec.~\ref{sec:periods} were recomputed from the exact cycle. For a subset of
the runs, $\TG$ was obtained from this re-simulation of the final configuration. The independent implementation was also used
to simulate lattices with $L=5,\dots,30$ (2000 realizations per size); it
detects the attractor exactly, by storing every configuration of the
trajectory until one repeats.

Table~\ref{tab:sweeps} lists the parameter sweeps. Each sweep varies one
parameter around the reference point $\alpha=1/3$, $r=1$,
$P_x=\mathrm{round}(N/3)$, with $L=8$ and $L=9$ ($N=64$, $81$), and a further
sweep varies $L$. The simulations were organized in series of up to
$10^4$ realizations per parameter value, with independent random initial
conditions; Appendix~\ref{app:study} describes the protocol and the verification, and the Supplemental Material \cite{SM} the statistical treatment. Runs at the reference point coming from different sweeps are
statistically equivalent and are pooled (Sec.~\ref{sec:S-pool} of the Supplemental Material), together
with those at $\alpha=3/8$, which defines the same rule when $r=1$; this gives
more than $4.5\times10^5$
realizations at the reference point for each of the two sizes analyzed in
detail. The number of runs for each parameter value is given in
Table~\ref{tab:sweeps}.

\begin{table*}[t]
\caption{Parameter sweeps of the main simulations
(about $10^7$ runs in total). Reference values: $\alpha=1/3$, $r=1$,
$P_x=\mathrm{round}(N/3)$. The $\alpha$ sweep does not contain $\alpha=1/3$
itself but $\alpha=3/8$, which defines the same rule for $r=1$; its runs are
counted at the reference point. In addition, \prelim{$5.2\times10^{4}$} runs of the
independent implementation cover $L=5,\dots,30$ (Sec.~\ref{sec:size} and
Appendix~\ref{app:study}).}
\label{tab:sweeps}
\footnotesize
\begin{ruledtabular}
\begin{tabular}{lllrrr}
Sweep & $L$ & Values & Runs per value & Total runs & Runs at reference\\
\colrule
threshold $\alpha$ & 8 & $m/8$, $m=0,\dots,8$ & \prelim{$1.1$--$1.3\times10^{5}$} & \prelim{$1.0\times10^{6}$} & \prelim{$1.1\times10^{5}$}\\
                   & 9 & $m/8$, $m=0,\dots,8$ & \prelim{$1.1$--$1.3\times10^{5}$} & \prelim{$1.0\times10^{6}$} & \prelim{$1.1\times10^{5}$}\\
radius $r$         & 8 & $1,2,3$ & \prelim{$1.1$--$1.2\times10^{5}$} & \prelim{$3.5\times10^{5}$} & \prelim{$1.2\times10^{5}$}\\
                   & 9 & $1,\dots,4$ & \prelim{$1.0$--$1.2\times10^{5}$} & \prelim{$4.4\times10^{5}$} & \prelim{$1.2\times10^{5}$}\\
composition $P_x$  & 8 & $0,\dots,21$ & \prelim{$1.1$--$1.3\times10^{5}$} & \prelim{$2.5\times10^{6}$} & \prelim{$1.1\times10^{5}$}\\
                   & 9 & $0,\dots,27$ & \prelim{$1.1$--$1.2\times10^{5}$} & \prelim{$3.1\times10^{6}$} & \prelim{$1.1\times10^{5}$}\\
size $L$           & $5,\dots,30$ & -- & \prelim{$2.2\times10^{4}$--$1.3\times10^{5}$} & \prelim{$1.7\times10^{6}$} & \\
\end{tabular}
\end{ruledtabular}
\end{table*}

\section{Results}
\label{sec:results}

\subsection{Consistency checks}

The independent implementation (Sec.~\ref{sec:methods}) reproduced exactly
the class of the attractor, the transient length and the five periods recorded by
the simulation code in all the \prelim{$2.4\times10^4$} runs that we re-simulated from
their initial condition, and $\Tst$, $\Tpop$, $\Tcol$ and $\Tgeo$ in all the
runs whose final configuration we re-simulated (Appendix~\ref{app:study}). The
population period reported by the code (the maximum of the periods of
$n_\mathrm{X}$, $n_\mathrm{Y}$ and $n_\mathrm{Z}$) always coincided with their
least common multiple, and the $S_3$-based uncolor period always coincided
with the $C_3$-based $\Tcol$ of Sec.~\ref{sec:periods}. None of the
divisibility relations (a)--(c) of Corollary~\ref{cor:lattice} was violated in
any of the \prelim{$1.2\times10^6$} oscillating runs, and $\Tdiv\mid\TG$ failed in
\prelim{21} runs on long cycles at $L\ge15$, in which the code recorded a
spurious diversity period (Appendix~\ref{app:study}).

\subsection{Census of attractors}

At the reference threshold $\alpha=1/3$ (a site needs at least three of its
eight neighbors in the predator state when $r=1$) most realizations end in a
fixed point (Fig.~\ref{fig:census}a): in \prelim{76--79\%} of the runs one state
invades the whole lattice (\emph{absorbing} state) and in \prelim{17--22\%} the
lattice freezes in a configuration in which no site has enough predators to
change. Without exception, frozen configurations contain exactly two states:
a predator--prey pair can be stable because the prey only sees its predator
along a short interface, but whenever the three states are present some site
sees enough predators and the configuration moves on. Most frozen
configurations are in fact ``almost absorbing'': a sea of one state containing
a small island of its predator that is too small to invade it. Up to the
symmetries in $\Gamma$, the \prelim{$8.2\times10^4$} frozen runs at $L=8$ end in
\prelim{359} distinct configurations and the \prelim{$7.7\times10^4$} at $L=9$ in
\prelim{534}; the most frequent one, a single predator site in the sea of
its prey, accounts for about half of them (\prelim{49\%} and \prelim{52\%}), followed by
a domino and a $2\times2$ block of predator sites. A prey site next to such an
island sees at most two predators among its eight neighbors, below the
threshold of three. Only \prelim{2--5\%} of the runs end in an oscillating
attractor (Table~\ref{tab:census}).

Increasing the radius to $r=2$ shifts the outcome from dominance to frozen
coexistence: the frozen fraction rises to \prelim{42\%} ($L=8$) and \prelim{39\%}
($L=9$), the absorbing fraction falls to \prelim{56--58\%}, and the dynamics
becomes faster (median transient of \prelim{5--7} steps against \prelim{14--19} for
$r=1$). Transients also grow with $L$: for $r=1$ the median time to absorption
goes from \prelim{6} steps at $L=5$ to \prelim{20} at $L=9$.

The extreme parameter values are trivial and serve as a check. For low
thresholds ($\alpha=0$ and $\alpha=1/8$) every run ends in a \emph{uniform
rotation} in which all sites change at every step, $s_{t+1}=\sigma s_t$, with
$\Tst=3$ and $\Tcol=1$. For high thresholds, $\alpha\ge1/2$ ($\theta\ge4$), no
front can advance: more than \prelim{99.5\%} of the runs freeze, for
$\alpha\ge5/8$ already in the initial configuration, and no oscillating attractor
is observed. For $P_x=0$ the lattice is uniform from the start, and for
$P_x=1$ the single Y site is invaded at once and the single X site remains
isolated among Z sites, which it cannot invade (one predator out of eight is
below the threshold): every run freezes.

\begin{table*}[t]
\centering
\caption{Outcome of the runs at the reference point $\alpha=1/3$, $P_x=N/3$
(all runs pooled, including $\alpha=3/8$ for $r=1$). $\tau_{1/2}$ is the median transient length (in steps) over all
runs, and $\langle n_3\rangle$ the mean number of sites whose closed
neighborhood contains the three states at the end of the run. Fixed points dominate at every size, and the oscillating fraction is largest at $L=9$ ($r=1$).}
\label{tab:census}
\small
\begin{tabular}{rrrrrrrr}
\toprule
$L$ & $r$ & Runs ($10^3$) & Absorbing & Frozen & Oscillating & $\tau_{1/2}$ & $\langle n_3\rangle$\\
\midrule
5 & 1 & 129 & 75.6\% & 22.3\% & 2.1\% & 5 & 0.13\\
6 & 1 & 129 & 78.2\% & 18.7\% & 3.2\% & 8 & 0.14\\
7 & 1 & 121 & 79.4\% & 17.5\% & 3.1\% & 11 & 0.09\\
8 & 1 & 456 & 79.0\% & 17.9\% & 3.1\% & 14 & 0.08\\
8 & 2 & 121 & 56.1\% & 42.1\% & 1.8\% & 5 & 0.69\\
8 & 3 & 108 & 61.6\% & 38.4\% & 0.0\% & 3 & 0.00\\
9 & 1 & 452 & 78.0\% & 17.1\% & 4.9\% & 19 & 0.37\\
9 & 2 & 112 & 58.0\% & 39.0\% & 3.0\% & 7 & 1.20\\
9 & 3 & 103 & 49.4\% & 50.5\% & 0.1\% & 4 & 0.06\\
\bottomrule
\end{tabular}
\end{table*}

\begin{figure*}[t]
\centering
\includegraphics[width=\textwidth]{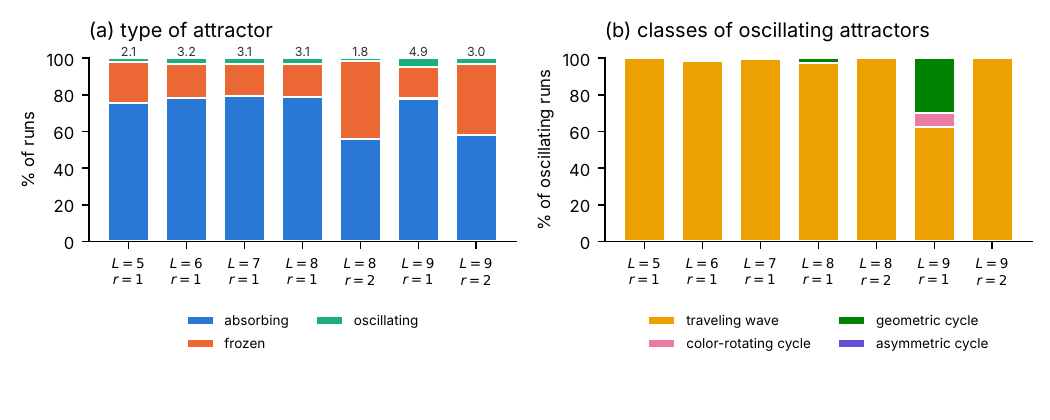}
\caption{Attractors reached from random initial conditions at $\alpha=1/3$,
$P_x=N/3$. (a) Fraction of runs ending in an absorbing state, a frozen
configuration or an oscillating attractor (the number above each bar is the
percentage of oscillating runs). (b) Classes of oscillating attractors,
defined by the tuple of periods (Sec.~\ref{sec:classes}); see Table~\ref{tab:sweeps} for the number of runs.}
\label{fig:census}
\end{figure*}

\subsection{Tuple frequencies and basins of attraction}\label{sec:basins}

Because the initial configurations are drawn uniformly among those with the
prescribed composition, the fraction of runs that end in a given attractor is an
unbiased estimate of the relative size of its basin of attraction within that
set of configurations. Grouping the runs by their tuple
$(\Tpop,\Tst,\Tcol,\Tgeo,\TG)$ therefore gives a coarse but symmetry-invariant
picture of the basin structure, and Table~\ref{tab:basins} lists it for the two
sizes on which we focus. With \prelim{$4.5\times10^5$} runs at each of $L=8$
and $L=9$ the statistical error of the larger basins is
below a tenth of a percent, small enough to compare $L=8$ and $L=9$ in
detail; the Good--Turing estimator (Sec.~\ref{sec:S-stats} of the Supplemental Material \cite{SM}) shows that
the list of tuples is complete for all practical purposes at these sample sizes.

Three points emerge. (i) The fixed points (absorbing plus frozen, all with the
trivial tuple) take \prelim{96.9\%} of the configuration space at $L=8$ and
\prelim{95.1\%} at $L=9$, with a very similar split between absorbing and
frozen states (\prelim{79.0/17.9\%} and \prelim{78.0/17.1\%}). (ii) At $L=8$ the oscillating part of the space is almost entirely
occupied by a single tuple, the traveling wave $(1,8,8,1,1)$, with four tiny
basins of geometric and color-rotating cycles (together \prelim{0.08\%}). At $L=9$ the oscillating fraction is
\prelim{1.6} times larger and is shared by \prelim{ten} tuples; the long cycle
$(24,72,72,24,8)$ alone has a basin (\prelim{1.0\%}) comparable to that of all
the waves of unequal widths (\prelim{2.7\%}). (iii) For the long cycles each tuple
corresponds to a single attractor (Sec.~\ref{sec:long}), so the
frequency of the tuple is directly the basin size of one attractor. The
waves and the frozen configurations, instead, gather several attractors that
the tuple does not distinguish (\prelim{9} waves at $L=8$ and \prelim{14} at $L=9$;
\prelim{359} and \prelim{534} frozen configurations up to $\Gamma$), so their
tuple frequency is the total basin of a family.

\begin{table}[t]
\caption{Frequency of each tuple $(\Tpop,\Tst,\Tcol,\Tgeo,\TG)$ at the
reference point ($\alpha=1/3$, $r=1$, $P_x=\mathrm{round}(N/3)$, including the
runs at $\alpha=3/8$, which define the same rule), all runs pooled. Since the initial conditions are drawn uniformly
among the configurations with the prescribed composition, the fraction of runs
ending in a tuple estimates the relative size of the union of the basins of
attraction of the attractors with that tuple. ``SE'' is the binomial standard
error of the percentage. Fixed points take almost all of the configuration space, and at $L=9$ the two-dimensional cycles together have a basin comparable to that of the traveling waves.}
\label{tab:basins}
\footnotesize
\begin{ruledtabular}
\setlength{\tabcolsep}{3pt}
\begin{tabular}{llrrr}
Attractor & Tuple & Runs & \% & SE\\
\colrule
\multicolumn{5}{l}{$L=8$ (\prelim{455\,559} runs)}\\
absorbing & $(1,1,1,1,1)$ & \prelim{359\,734} & \prelim{79.0} & 0.060\\
frozen & $(1,1,1,1,1)$ & \prelim{81\,647} & \prelim{17.9} & 0.057\\
traveling wave & $(1,8,8,1,1)$ & \prelim{13\,801} & \prelim{3.0} & 0.025\\
geometric & $(1,16,16,4,4)$ & \prelim{239} & \prelim{0.052} & 0.003\\
geometric & $(22,88,88,22,22)$ & \prelim{95} & \prelim{0.021} & 0.002\\
geometric & $(6,12,12,6,6)$ & \prelim{32} & \prelim{0.007} & 0.001\\
color-rotating & $(6,12,4,6,2)$ & \prelim{11} & \prelim{0.002} & 0.0007\\
\colrule
\multicolumn{5}{l}{$L=9$ (\prelim{451\,819} runs)}\\
absorbing & $(1,1,1,1,1)$ & \prelim{352\,412} & \prelim{78.0} & 0.062\\
frozen & $(1,1,1,1,1)$ & \prelim{77\,165} & \prelim{17.1} & 0.056\\
traveling wave & $(1,9,9,1,1)$ & \prelim{12\,067} & \prelim{2.7} & 0.024\\
geometric & $(24,72,72,24,8)$ & \prelim{4545} & \prelim{1.0} & 0.015\\
wave (equal bands) & $(1,9,3,1,1)$ & \prelim{1795} & \prelim{0.40} & 0.009\\
color-rotating & $(18,18,6,18,6)$ & \prelim{1715} & \prelim{0.38} & 0.009\\
geometric & $(21,189,189,21,21)$ & \prelim{1271} & \prelim{0.28} & 0.008\\
geometric & $(23,69,69,23,23)$ & \prelim{582} & \prelim{0.13} & 0.005\\
geometric & $(18,108,108,18,6)$ & \prelim{189} & \prelim{0.042} & 0.003\\
geometric & $(18,54,54,18,6)$ & \prelim{39} & \prelim{0.009} & 0.001\\
asymmetric & $(14,14,14,14,14)$ & \prelim{20} & \prelim{0.004} & 0.0010\\
geometric & $(21,63,63,21,7)$ & \prelim{19} & \prelim{0.004} & 0.0010\\
\end{tabular}
\end{ruledtabular}
\end{table}

\subsection{Traveling waves}\label{sec:waves}

Traveling waves account for \prelim{83\%} of all oscillating runs at
$\alpha=1/3$ and $L\le9$, and for practically all of them when $L\le 8$ or $r\ge2$. For $r=1$
every wave consists of three straight bands, one per state, arranged in the
cyclic order X, Y, Z and moving by one site per step perpendicular to the
bands (Fig.~\ref{fig:spacetime}a of the Supplemental Material \cite{SM}). The mechanism is simple and explains why the
waves are tied to the threshold: a site in the column adjacent to the band of
its predator has exactly three of its eight neighbors in the predator state,
$3/8\ge 1/3$, and changes; a site in the interior of its band has no predator
neighbor and does not change. Because the bands follow the cyclic order, the
three fronts advance in the same direction at the same speed and the width of
every band is conserved. This is precisely the band test of Fisch, Gravner and
Griffeath \cite{Fisch1991a}: with $\rho=1$ a front sees $2\rho+1=3$ predator
neighbors, so straight waves exist for $\theta\le3$, i.e.\ $\alpha\le3/8$, and
should disappear for larger thresholds. Hence only the three front columns change at each step
(a fraction $3/L$ of the sites: 60\% for $L=5$, 33\% for $L=9$), the
populations are constant, and the wave is completely characterized by its band
widths $(w_1,w_2,w_3)$, $w_1+w_2+w_3=L$. Up to the symmetries in $\Gamma$ we
found \prelim{2, 4, 5, 9 and 14} distinct waves for $L=5,\dots,9$; the most
frequent are those with the most balanced widths, e.g.\ 2-2-1 for $L=5$
(\prelim{2282} of \prelim{2765} runs). When $L$ is a multiple of three the equal-width wave
($w_i=L/3$) has the additional color symmetry $\Tcol=L/3$. Topologically,
a band wave winds once through the cycle X$\to$Y$\to$Z along the closed loop of
the torus perpendicular to the bands; in excitable automata on graphs such
non-contractible windings are known to sustain periodic waves
\cite{Cai2018,Doba2018}, which suggests that the winding number of the final
configuration along the two cycles of the torus could serve as an additional,
conserved label of the attractor. Point defects, i.e.\ plaquettes around which
the color winds (Sec.~\ref{sec:t24}), are absent from all the waves; whether
the winding numbers along the torus are conserved by the threshold rule is left
for future work. Waves form
quickly: the median transient grows from \prelim{4} steps ($L=5$) to
\prelim{11} ($L=9$), shorter than that of absorbing runs of the same size.

For $r=2$ the waves still move by one site per step with $\Tst=L$ and
$\Tgeo=1$, but only \prelim{31\%} ($L=8$) and \prelim{8\%} ($L=9$) of them have
straight fronts; the rest have jagged fronts, which multiplies the number of
distinct waves (\prelim{165} different waves, up to $\Gamma$, in
\prelim{2146} runs for $L=8$, and \prelim{761} in \prelim{3427} runs for $L=9$).
One run at $L=8$, $r=2$ ended in a wave whose jagged front only repeats its shape
every two steps, $(1,8,8,2,2)$.

\subsection{Long cycles}\label{sec:long}

The remaining \prelim{17\%} of the oscillating runs (\prelim{8380} runs at
$L=9$, $r=1$, \prelim{377} at $L=8$, \prelim{10} at $L=7$, \prelim{60} at $L=6$ and none
at $L=5$) end in genuinely two-dimensional
cycles: patches of the three states that chase each other, each state invading
the one in front of it while being invaded from behind, in the manner of the
spiral waves of spatial rock--paper--scissors games but on a torus so small that
the wave collides with itself (Fig.~\ref{fig:long}). Table~\ref{tab:long} of the Supplemental Material
summarizes them. Three features stand out.

First, each period corresponds to essentially a single attractor. Grouping the
final cycles up to the elements of $\Gamma$ and time shifts, the
\prelim{4545} runs with $\Tst=72$ all end in the same cycle, and the same
happens for every other long cycle at $L=9$ except the color-rotating one with
$\Tst=18$, which has two variants; at $L=6$, $7$ and $8$ there are at most five
distinct cycles per tuple (Table~\ref{tab:long} of the Supplemental Material).

Second, most long cycles drift: after $\Tgeo$ steps the pattern reappears
translated by a few sites, so that $\Tst$ is $\Tgeo$ times the number of
translations needed to go round the torus. The drift is much slower than the
waves: $1/8$ site per step for $\Tst=72$ and $4/21\approx0.19$ for
$\Tst=189$, against one site per step. The 18-step cycle does not drift at all; it
pulsates in place and returns with its colors rotated every six steps.

Third, the populations show strong cycles of dominance, absent in the waves
(Fig.~\ref{fig:pops}). In the 72-step cycle each state goes from \prelim{10} to
\prelim{56} of the 81 sites and dominates once every 24 steps; in the 189-step
cycle it goes from 12 to 53 every 21 steps; in the pulsating 18-step cycles the
populations change much less (20--37). Long cycles also take longer to form
(median transients of \prelim{20--36} steps for the most frequent cycles against \prelim{11} for the waves at
$L=9$).

The abundance of long cycles at $L=9$ (\prelim{1.9\%} of the runs, against
\prelim{0.08\%} at $L=8$ and less than \prelim{0.05\%} at $L\le7$) is striking.
Several of the transformations that generate them have order three (a
translation by three sites, a color rotation), which is possible only because
$L=9$ is a multiple of three. As shown in Sec.~\ref{sec:size}, $L=9$ is also
the smallest size at which the population period $\Tpop=24$ appears, and this
period is present in every larger size of our sample.

\begin{figure*}[t]
\centering
\includegraphics[width=\textwidth]{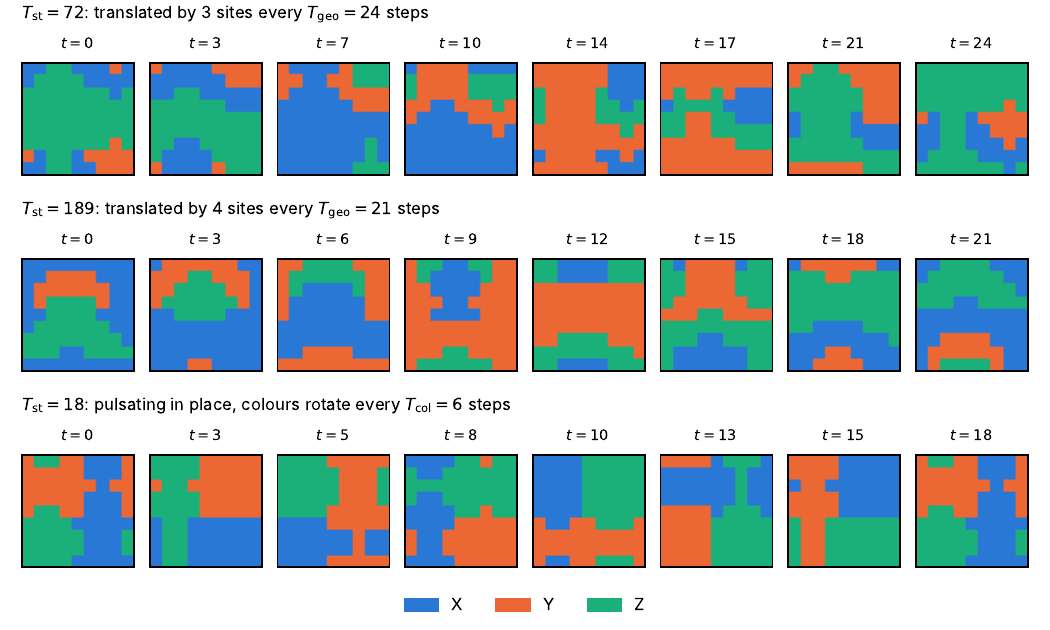}
\caption{Three long cycles at $L=9$, $r=1$, $\alpha=1/3$, shown over one
geometric period. Top: $\Tst=72$, the pattern returns translated by three
sites after $\Tgeo=24$ steps. Middle: $\Tst=189$, translated by four sites after
21 steps. Bottom: $\Tst=18$, the pattern pulsates in place and returns with its
colors rotated after $\Tcol=6$ steps.}
\label{fig:long}
\end{figure*}

\subsection{Local diversity}\label{sec:diversity}

Cyclic dominance is usually studied as a mechanism that sustains diversity
\cite{Kerr2002,Reichenbach2007}, so it is natural to ask where local diversity,
i.e.\ the coexistence of the three states within a neighborhood, survives in
this model. In the final
state, absorbing and frozen configurations have no site with the three states
in its closed neighborhood ($n_3=0$), so that the average over all runs is
of the order of one site or less (Table~\ref{tab:census}). Local diversity survives only on
oscillating attractors, and very unevenly: traveling waves keep \prelim{2--6}
such sites on average for $r=1$ (only those next to bands of width one), the long
cycles keep \prelim{14--22} in their final configuration, because the patches touch each other on several
sides, and waves with $r=2$ keep \prelim{about 39}, simply because larger
neighborhoods reach across the bands. For $r=1$ the long cycles are therefore the
attractors that best preserve local diversity.

The symmetry analysis of Sec.~\ref{sec:periods} explains where this diversity
sits and how balanced it is, through three exact links.

\paragraph{Diversity is an intrinsic observable.} The diversity counts $n_m$
are invariant under the whole of $\Gamma$, so they are functions on the quotient
space and their period always divides $\TG$ (Corollary~\ref{cor:lattice}d),
which holds in all runs except \prelim{21} on three long cycles at $L\ge15$, where
the code recorded a spurious diversity period (Appendix~\ref{app:study}). Local
diversity is therefore blind to the lattice-induced factor $m$: a drifting and a
stationary member of the same family, with very different $\Tst$, have the same
diversity dynamics.

\paragraph{Defects guarantee diversity.} A plaquette of $2\times2$ sites with
winding number $\pm1$ (Sec.~\ref{sec:t24}) must contain the three states,
because the four steps around it, each in $\{-1,0,1\}$, add up to $\pm3$; and the
closed neighborhood of each of its four sites contains the whole plaquette.
Every topological defect thus creates at least four sites with $n_3=1$, and
these sites cannot disappear while the defect persists. The data show that
this is the main source of local diversity on two-dimensional attractors. At
$L=9$ every long or color-rotating cycle contains two or four defects, and the
defect plaquettes account for \prelim{48--93\%} of the sites with three states
(\prelim{14--18} sites with three states on average over a whole cycle, about a fifth of the lattice), the rest lying on the
junctions between patches. Traveling waves have no defects, and their few
diverse sites lie on bands of width one. For $L=31$--40 the share of the lattice
with three states falls to \prelim{1--2\%}, and in the stationary $\Tpop=24$
attractors more than \prelim{99\%} of these sites belong to the spiral cores: the
number of diverse sites is set by the number of defects (typically two or four),
not by the area. On large lattices local diversity is therefore a
topological property, localized at the cores of the spiral pairs.

\paragraph{Color symmetry guarantees balanced coexistence.} If $\Tcol<\Tst$,
then $s_{t+\Tcol}=\sigma^c s_t$ with $c\neq0$ and $\Tst=3\Tcol$, so
$n_k(t+\Tcol)=n_{k-c}(t)$ for the population $n_k$ of each state, and the
averages of the three populations over one state period are equal, as are the
time averages of any other observable that is permuted by $\sigma$. All
attractors with a color symmetry in our data have exactly balanced average
populations. Without it the balance is not guaranteed: traveling waves with
bands of unequal width keep the three states with average densities that differ
by up to \prelim{1/3}, and the long cycles without color symmetry, such as
$(21,189,189,21,21)$, $(23,69,69,23,23)$ and the asymmetric cycle, differ by
\prelim{3--7\%}. The type of the generator $h$ thus determines not only whether
the three states coexist but whether they do so on equal terms.

\subsection{Classes of oscillating attractors}\label{sec:classes}

Table~\ref{tab:catalogue} of the Supplemental Material \cite{SM} lists all the tuples
$(\Tpop,\Tst,\Tcol,\Tgeo,\TG)$ observed for oscillating attractors at
$\alpha=1/3$, together with the transformation that generates the cycle.
Representative snapshots are shown in Fig.~\ref{fig:snapshots} and the
corresponding population time series in Fig.~\ref{fig:pops}. The tuples fall
into five classes, which we define directly from the periods:

\begin{description}
\item[Uniform rotation] ($\Tcol=1$). The whole configuration changes color
at every step. The state period is $\Tst=3$ and the pattern of domains never
changes. The population period distinguishes two cases: $\Tpop=3$ when the
three populations are different, and $\Tpop=1$ when they are exactly equal,
as happens for $L=9$, $\alpha=0$ where $n_c=27$ for all $c$. Thus, at $\alpha=0$,
$L=8$ and $L=9$ give the tuples $(3,3,1,3,1)$ and $(1,3,1,3,1)$: the same
microscopic dynamics, distinguishable only through $\Tpop$.

\item[Traveling waves] ($\Tgeo=1<\Tst$). Bands move rigidly by one site per
step (Fig.~\ref{fig:snapshots}a). The configuration returns after $\Tst=L$ steps,
but since every configuration is a translate of the previous one, the
populations are constant and $\Tpop=1$. When the three bands have equal width
($L$ divisible by 3), a shift by $L/3$ sites is equivalent to a color
rotation and $\Tcol=L/3$; otherwise $\Tcol=L$. Traveling waves are by far the
most common oscillating attractor (Sec.~\ref{sec:waves}).

\item[Color-rotating cycles] ($\Tcol<\Tst$, $\Tgeo>1$). After $\Tcol=\Tst/3$
steps the configuration is the same pattern with the colors rotated
(Fig.~\ref{fig:snapshots}b). The three populations then follow the same
time series shifted by $\Tcol$ (Fig.~\ref{fig:pops}b), so that $\Tpop=\Tst$
while the sorted populations repeat every $\Tcol$ steps. Example: $L=9$,
$(18,18,6,18,6)$. For $L=6$ and $L=8$ we also found $(6,12,4,6,2)$
(Fig.~\ref{fig:snapshots}c), in which the color rotation combined with a
translation by $(3,3)$ (by $(0,4)$ for $L=8$) maps the configuration onto
itself after only $\TG=2$ steps.

\item[Geometric cycles] ($\Tgeo<\Tst=\Tcol$). After $\Tgeo$ steps the pattern
reappears translated (and sometimes reflected) and the full period is
$\Tst=m\,\Tgeo$, with $m$ the order of the transformation. These are the long
cycles: $(21,189,189,21,21)$ ($L=9$, translation by 4 sites, $m=9$),
$(24,72,72,24,8)$ ($L=9$, translation by 3 sites, $m=3$),
$(23,69,69,23,23)$, $(18,108,108,18,6)$ (reflection plus translation, $m=6$),
$(22,88,88,22,22)$ ($L=8$, $m=4$). Here the populations repeat after $\Tgeo$
steps, so $\Tpop=\Tgeo<\Tst$, and an observer with access only to census
data would underestimate the period by the factor $m$.
In several of these cycles $\TG<\Tgeo$: for $(24,72,72,24,8)$ the configuration
after 8 steps is the initial one translated by two sites \emph{and} with its
colors rotated (Fig.~\ref{fig:snapshots}d). This elementary motif of 8 steps
is invisible to $\Tcol$ (72) and to $\Tgeo$ (24), and in the population series
it appears only as the threefold repetition of the same peak shape with the
colors permuted (Fig.~\ref{fig:pops}d). A special case is $(1,16,16,4,4)$
($L=8$), where the transformation is a rotation by $90^\circ$ combined with a
translation by $(4,4)$ and the
populations are constant.

\item[Asymmetric cycles] ($\Tpop=\Tst=\Tcol=\Tgeo=\TG$, generator $h=\mathrm{id}$).
No element of $\Gamma$ maps the configuration onto a later one before it returns
to itself, so all periods coincide and any one of them is a complete description.
At $L=8$ and $L=9$ there is a single case, $(14,14,14,14,14)$ at $L=9$, described
below.
\end{description}

\paragraph{The asymmetric cycle at $L=9$.} The tuple $(14,14,14,14,14)$ was reached
in \prelim{20} runs at $L=9$, about \prelim{0.1\%} of the oscillating runs, and all of
them belong to one attractor up to $\Gamma$. Its structure is shown in
Fig.~\ref{fig:asym}. The configuration contains two pairs of opposite spiral
defects, the same objects that build the $\Tpop=24$ family of Sec.~\ref{sec:t24},
but here they are pinned: the four defects stay at the same plaquettes during the
whole cycle, and the domains rotate around them. Every site has period
$T_i=14$ and makes exactly one turn per cycle, so the whole lattice rotates as
a single phase field with frequency $3/14\approx0.21$, faster than the
$\Tpop=24$ spirals ($1/8$). The populations oscillate with large amplitude
(between 13 and 36 sites per state) and without any repetition within the cycle,
which is why $\Tpop=\Tst$. The name ``asymmetric'' refers to the dynamics, not to
the pattern: each configuration of the cycle is invariant under a reflection
that exchanges the two pairs of defects, but this static symmetry does not map
the configuration onto a later one, and no translation, rotation or color
permutation does either. The same tuple is reached on the $9\times12$ torus
(Sec.~\ref{sec:tests}), as expected for a pattern that uses only the side of
length 9.

Asymmetric cycles are rare at the sizes of the main study but not in general
(Fig.~\ref{fig:asym}c). They are absent for $L\le8$, take less than
\prelim{4\%} of the oscillating runs for $L\le18$, and become common for $L\ge19$,
because the attractors with $\Tpop=80$, which take up to \prelim{38\%} of the
oscillating runs at $L=30$, belong to this class: they return to themselves after 80
steps without any intermediate symmetry. In this sense the $(14,14,14,14,14)$
cycle of $L=9$ is the smallest representative of the class that dominates the
oscillating attractors of large lattices together with the $\Tpop=24$ family.

\begin{figure*}[t]
\centering
\includegraphics[width=\textwidth]{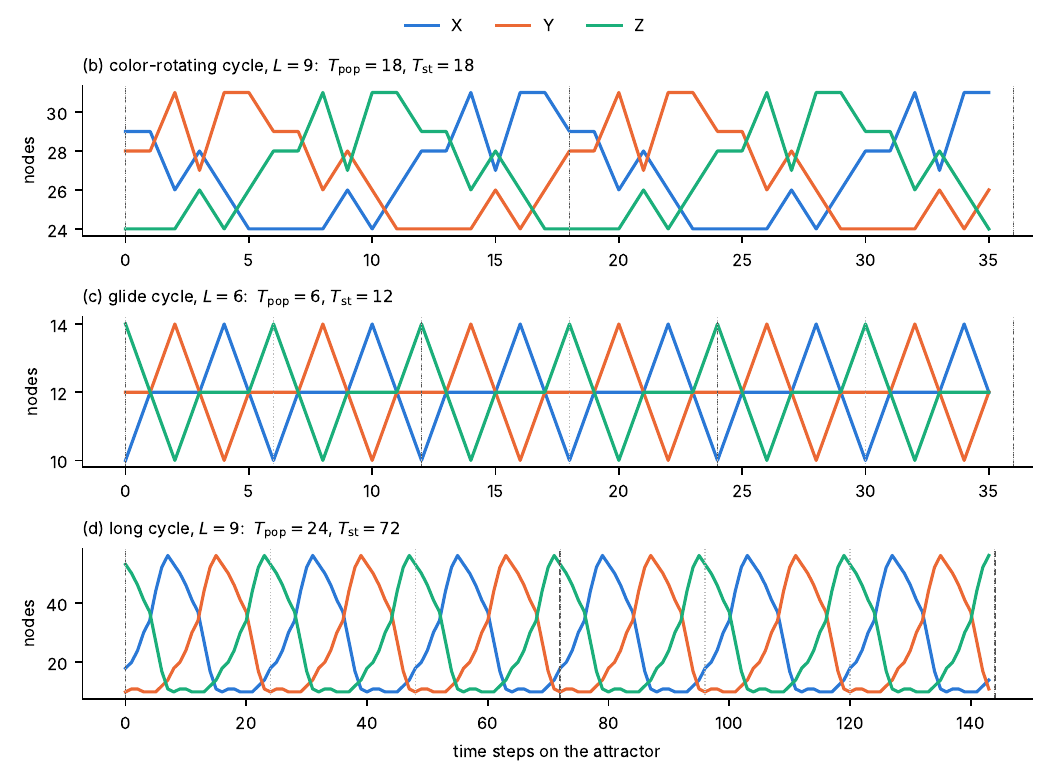}
\caption{Populations $n_\mathrm{X}$, $n_\mathrm{Y}$, $n_\mathrm{Z}$ on the
attractors (b)--(d) of Fig.~\ref{fig:snapshots}. Dashed lines mark multiples of
$\Tst$ and dotted lines multiples of $\Tpop$. In (c) and (d) the populations
repeat several times within one state period; in (b) and (d) the three curves
are copies of each other shifted in time, the signature of a color rotation
in $\Tcol$ or $\TG$.}
\label{fig:pops}
\end{figure*}

\begin{figure*}[t]
\centering
\includegraphics[width=\textwidth]{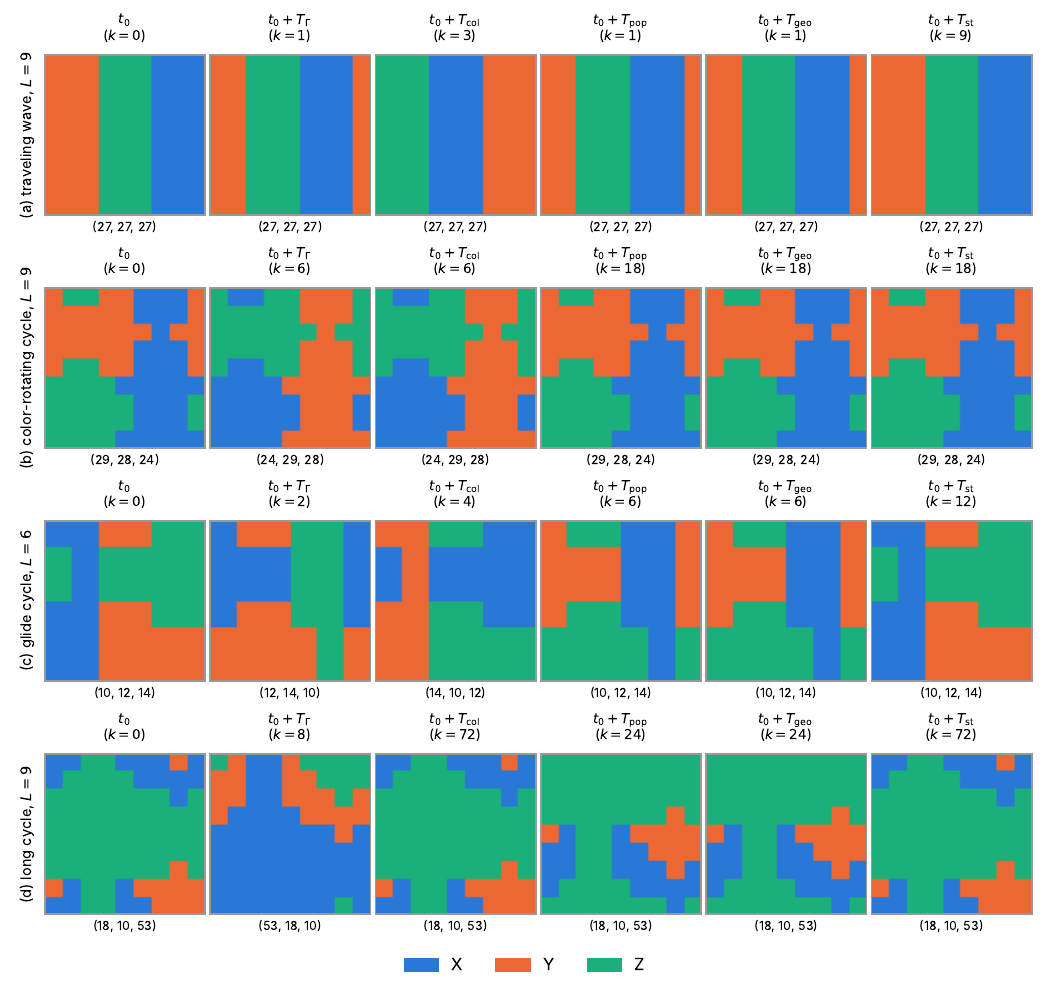}
\caption{Snapshots of four oscillating attractors ($\alpha=1/3$, $r=1$) at
the time lags given by the periods of the tuple. In each row the first panel
is an arbitrary configuration of the cycle and the following ones are the
configurations after $\TG$, $\Tcol$, $\Tpop$, $\Tgeo$ and $\Tst$ steps. The
numbers under each panel are the populations $(n_\mathrm{X},n_\mathrm{Y},n_\mathrm{Z})$:
they coincide with those of the first panel after $\Tpop$ steps, and are
cyclically permuted after a color rotation.
(a) Traveling wave, $L=9$: $(\Tpop,\Tst,\Tcol,\Tgeo,\TG)=(1,9,3,1,1)$.
(b) Color-rotating cycle, $L=9$: $(18,18,6,18,6)$.
(c) Color-rotating cycle with translation, $L=6$: $(6,12,4,6,2)$.
(d) Long geometric cycle, $L=9$: $(24,72,72,24,8)$.}
\label{fig:snapshots}
\end{figure*}

\begin{figure*}[t]
\centering
\includegraphics[width=0.85\textwidth]{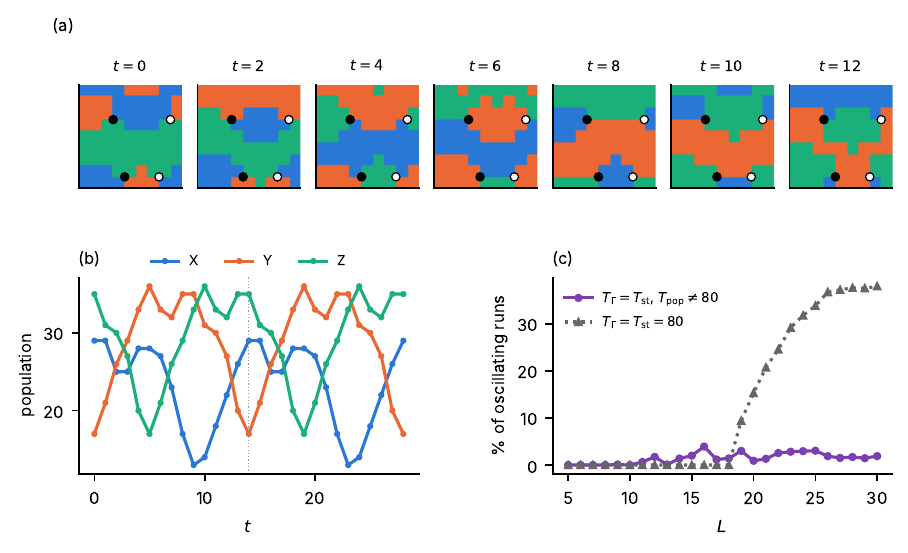}
\caption{The asymmetric cycle $(14,14,14,14,14)$ at $L=9$ ($\alpha=1/3$, $r=1$,
$P_x=27$). (a) Configurations every two steps over one period; white and black
circles mark the plaquettes with winding number $+1$ and $-1$ (spiral cores),
which do not move. (b) Populations of the three states over two periods (dotted
line: $t=\Tst=14$). (c) Share of the oscillating runs that end in an asymmetric
cycle ($\TG=\Tst$, $h=\mathrm{id}$) as a function of $L$, separating the
attractors with $\Tpop=80$.}
\label{fig:asym}
\end{figure*}

\subsection{What a single period misses}

To quantify the complementarity of the periods we counted, among the
oscillating attractors at $\alpha=1/3$ with $L\le9$ (Table~\ref{tab:catalogue} of the Supplemental Material,
all $r$ pooled), how many distinct tuples share each value of a single period.
There are \prelim{26} distinct tuples. Knowing only $\Tpop$ leaves \prelim{13}
distinguishable values, and the value $\Tpop=1$ alone gathers \prelim{9}
different tuples (all the traveling waves plus the rotated cycle of $L=8$).
Knowing only $\Tgeo$ or only $\TG$ leaves \prelim{14} and \prelim{12} values, with up
to 7 tuples merged. $\Tst$ alone gives \prelim{18} values: for instance $\Tst=9$ does not
distinguish the equal-band from the unequal-band waves of $L=9$, nor these
from the color-rotating cycles of $L=6$ and $L=7$ with $\Tst=9$. $\Tcol$ alone
comes closest (\prelim{20} values). The pairs $(\Tcol,\Tgeo)$ and $(\Tcol,\TG)$
separate all \prelim{26} tuples; $(\Tpop,\Tcol)$ fails only for the two waves of
$L=8$, $r=2$ with straight and with breathing fronts, and pairs without $\Tcol$
give at most \prelim{24} classes. No single period is therefore a complete
label even in this small catalogue, and the remaining periods are needed to
interpret the label: $\Tpop$ gives the link with aggregated data, $\Tgeo$ the
motion of the pattern and $\TG$ the hidden space--color motif.

\subsection{Dependence on lattice size}\label{sec:size}

We focus on $L=8$ and $L=9$ because they are large enough to display the
whole repertoire of attractor classes and small enough to sample the basins
with high precision, re-simulate every attractor and test all elements of
$\Gamma$. The size $L=9$ has the additional singularity of being a multiple of
three, which allows the equal-band wave, translations of order three and
attractors whose symmetry combines a translation with a color rotation.
Larger lattices are nevertheless informative about how representative these
two sizes are. Besides the $L$ sweep of the simulation code (up to $L=30$, $10^5$ runs per
size for $L\le14$, \prelim{$1.2\times10^5$} at $L=15$ and 16, \prelim{$5.1\times10^4$} at $L=17$ and \prelim{$2.2\times10^4$} for
larger sizes) we ran \prelim{2000} realizations per size for
$L=5,\dots,30$ with the independent implementation of Sec.~\ref{sec:methods}
(Fig.~\ref{fig:size}); for $L\le9$ its frequencies agree with those of the
main simulations within statistical error. Larger square tori ($L=31$--40) and a
rectangular torus are discussed in Sec.~\ref{sec:tests}.

The frequency of the fixed points decreases steadily with $L$
(Fig.~\ref{fig:size}a). It is almost constant up to $L=9$ (\prelim{95--98\%}),
falls below one half at \prelim{$L\approx20$} and is only \prelim{16\%} at $L=30$,
where \prelim{84\%} of the runs end in an oscillating attractor. Both kinds of
fixed point lose weight, but the frozen configurations do so more slowly up to
$L\approx18$: the small islands that freeze the lattice are local objects, and
the probability that a random configuration produces one before being invaded
decreases only slowly with the area. The transients grow accordingly, with
median lengths of \prelim{19} steps at $L=9$ and \prelim{320} at $L=30$.

The oscillating attractors change character with $L$
(Fig.~\ref{fig:size}b). Traveling waves ($\Tpop=1$, $\Tst=L$) exist at every
size, but their share of the runs peaks around \prelim{$L\approx21$} and then
decreases. Two population periods take over at large sizes:
$\Tpop=24$, which appears for the first time at $L=9$, in the long cycle
$(24,72,72,24,8)$, and is present in every larger size of our sample, and
$\Tpop=80$ (with $\Tst=\TG=80$), which appears at \prelim{$L=19$} and becomes as
frequent as $\Tpop=24$ for $L\ge24$. The first of them is the subject of the
next section.

\begin{figure*}[t]
\centering
\includegraphics[width=\textwidth]{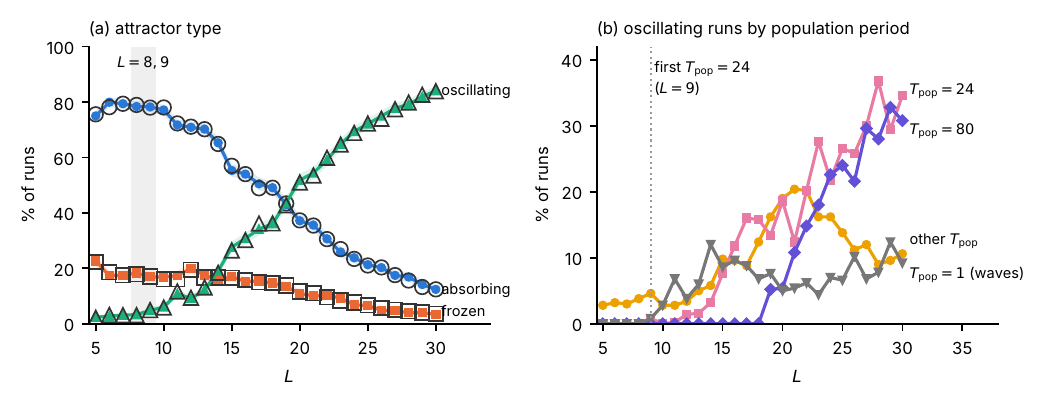}
\caption{Dependence on the lattice size at $\alpha=1/3$, $r=1$,
$P_x=\mathrm{round}(N/3)$. (a) Fraction of runs ending in an absorbing state,
a frozen configuration or an oscillating attractor; lines and filled symbols:
independent Python simulation, \prelim{2000} runs per size (bands: $\pm2$
standard errors); open symbols: all runs of the main simulations for
$L\le9$. The shaded band marks the sizes analyzed in detail in this work.
(b) Fraction of runs ending in an oscillating attractor with a given
population period (\prelim{500} runs per size). The population period 24
appears first at $L=9$ and persists for larger sizes.}
\label{fig:size}
\end{figure*}

\subsection{Intrinsic and lattice-induced periods: the $\Tpop=24$ family}\label{sec:t24}

The attractors with $\Tpop=24$ are the clearest illustration of
Proposition~\ref{prop:factor}, and the reason why the comparison between $L=8$
and $L=9$ is more than a comparison between two sizes: the family is absent at
$L=8$, appears at $L=9$ (\prelim{1.0\%} of the runs) and is found at every larger
size up to $L=30$, where it accounts for \prelim{35\%} of the runs
(Fig.~\ref{fig:t24}f), and beyond, up to $L=40$ (Sec.~\ref{sec:tests}). Across the
22 lattices with $9\le L\le30$ we observed \prelim{82} combinations of size and tuple
with $\Tpop=24$ (\prelim{45} distinct tuples), with state periods between 24 and 696
(Fig.~\ref{fig:t24}e). For every one of them we re-simulated the attractor and
determined the element $h$ of $\Gamma$ that maps the configuration onto itself
after $\TG$ steps (Table~\ref{tab:t24}). The result is remarkably simple.

\paragraph{What does not depend on $L$.} All members of the family have
$\Tpop=24$ and $\TG\in\{8,24\}$, and $h$ is of only four types:
\begin{enumerate}
\item $h=\sigma$, the tuple $(24,24,8,24,8)$: a \emph{stationary} cycle in which,
every 8 steps, the configuration is the previous one with all the colors
advanced by one ($s_{t+8}=\sigma s_t$). From $L=14$ on, in all the attractors that we
re-simulated, every site of the lattice changes state exactly once every 8 steps, so that the whole configuration
behaves as a phase field rotating with period 24: all site periods equal 24 and
every site makes exactly one turn per state period (Sec.~\ref{sec:sites}). This is the dominant member
from $L=15$ on.
\item $h=\mathrm{id}$, the tuple $(24,24,24,24,24)$: a stationary variant without
the color symmetry, much less frequent (\prelim{2--8\%} of the stationary runs for
$L\ge14$).
\item $h=\sigma\circ\tau_v$ with $\TG=8$: a \emph{drifting} cycle that advances
its colors by one and moves by $v$ every 8 steps. For $L\ge11$ (except a few runs
at $L=24$), $v$ is a single lattice step, $(1,0)$ or $(1,1)$; at $L=9$ and 10
it is $(2,0)$.
\item $h=\tau_{(3,3)}$ with $\TG=24$: a drifting variant without the color
symmetry, which moves by three diagonal steps every 24 steps, i.e.\ with the
same velocity as $\sigma\circ\tau_{(1,1)}$.
\end{enumerate}
Two exceptions, each with a few runs, complete the picture: a cycle at $L=13$
whose generator is a glide reflection, and drifting cycles at $L=15$ and $L=18$
with $h=\tau_{(0,1)}$ and $\tau_{(0,2)}$ after 24 steps.

The spatial structure is equally stable. Counting the plaquettes of $2\times2$
sites around which the color winds once in the cyclic order (winding number
$+1$) or in the opposite order ($-1$), every attractor of the family that we
re-simulated contains one pair of opposite defects, or occasionally two or three
pairs. These are the cores
of two counter-rotating spirals, clearly visible at $L=30$ (Fig.~\ref{fig:t24}d)
and still recognizable, squeezed by the boundary conditions, at $L=9$
(Fig.~\ref{fig:t24}a). The same count gives one defect pair for the
$\Tpop=80$ attractors and for most of the long cycles at $L=9$, and none for the
traveling waves, whose winding is around the torus rather than around a point.
The period 24 can be read as the rotation period of the spiral pair, and
$\TG=8$ as the time in which each site advances by one color, i.e.\ one third of
a turn. Why this rotation period takes precisely the value 24
remains an open question. No analytical explanation is available, but the value
is robust: it is the same at every size from $L=9$ to $L=40$, on the rectangular
torus, in the stationary and drifting members, with and without color symmetry,
and in the two independent implementations, which suggests an organizing
mechanism local to the spiral core.

\paragraph{What depends on $L$.} The state period is the only member of the
tuple, apart from $\Tcol$ and $\Tgeo$, which inherit it, that changes with the
lattice, and Proposition~\ref{prop:factor} predicts it exactly. For the
stationary members $m=3$ ($h=\sigma$) or $m=1$, and $\Tst=24$ at every size. For the
generic drift by one lattice step per 8 steps, $h=\sigma\circ\tau_v$ with
$\gcd(L,v_1,v_2)=1$, Eq.~\eqref{eq:order} gives $m=\mathrm{lcm}(3,L)$ and
\begin{equation}
\Tst=8\,\mathrm{lcm}(3,L)=\frac{24\,L}{\gcd(L,3)},
\label{eq:t24}
\end{equation}
which is 72 at $L=9$, 264 at $L=11$, 96 at $L=12$ and 696 at $L=29$; for
$h=\tau_{(3,3)}$, $m=L/\gcd(L,3)$ and Eq.~\eqref{eq:t24} holds again. The relation
$\Tst=\TG\,m$ holds without exception in all \prelim{82} members of the family,
including the exceptional ones (Table~\ref{tab:t24}): the drift by two sites at
$L=10$ gives $m=\mathrm{lcm}(3,5)=15$ and $\Tst=120$ instead of 240, the glide
reflection at $L=13$ has $m=26$ and $\Tst=624$, and so on. Equation~\eqref{eq:t24}
is therefore not a general law but the special case of Eq.~\eqref{eq:order} for
this generator: it accounts for \prelim{98.7\%} of the drifting runs, and the
remaining ones, whose generator is a drift by more than one step, a glide
reflection or a pure translation, follow Eq.~\eqref{eq:order} with their own $h$.

The family therefore shows, in a single object followed across 22 lattice sizes
(32 with those of Sec.~\ref{sec:tests}),
the separation that the tuple of periods is designed to make. $\Tpop=24$, $\TG$,
the type of symmetry and the drift velocity are properties of the pattern; the
state period, which ranges from 24 to 696 for $L\le30$ (and up to 960 at
$L=40$, Sec.~\ref{sec:tests}), is a property of the pattern
\emph{embedded in a particular torus}, and the same pattern is stationary or
drifting depending on whether it fits the lattice. A study based only on $\Tst$
would report a different attractor at every size, and a study based only on
$\Tpop$ would not distinguish the drifting from the stationary members; the
combination of $\TG$, $h$ and Proposition~\ref{prop:factor} identifies them as
one family and explains every value of $\Tst$. The appearance of the family at
$L=9$ and not at $L=8$, and its growth with $L$, suggest that a minimum area is
needed to host the spiral pair: at $L=9$ it is squeezed into a drifting
configuration, and only from $L\approx14$ does it fit in a stationary form.

\begin{figure*}[t]
\centering
\includegraphics[width=\textwidth]{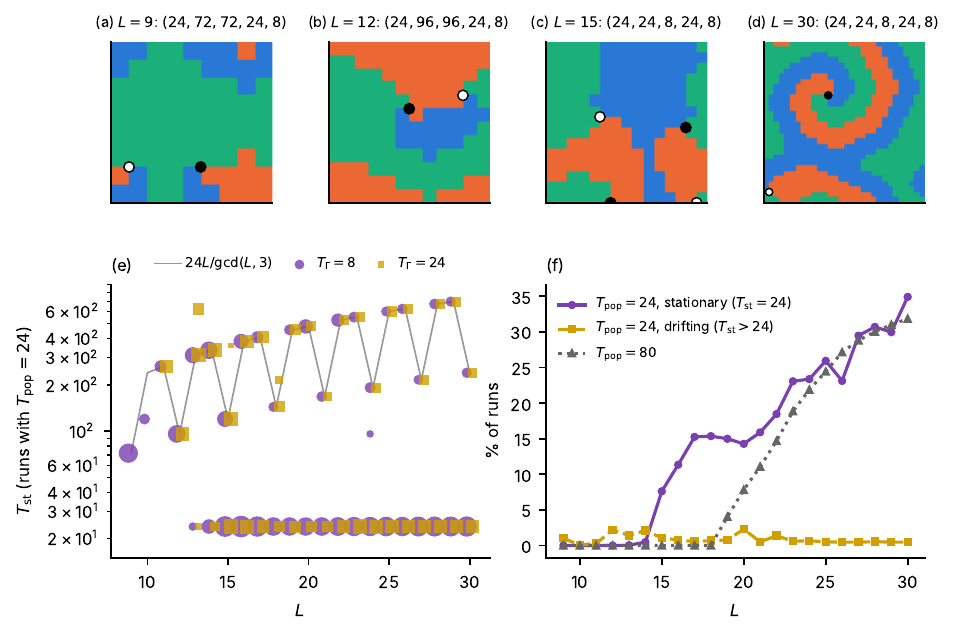}
\caption{The $\Tpop=24$ family ($\alpha=1/3$, $r=1$, $P_x=\mathrm{round}(N/3)$).
(a)--(d) One configuration of a member of the family at $L=9$, 12, 15 and 30; white
and black circles mark the plaquettes with winding number $+1$ and $-1$ (spiral
cores). (a) and (b) drift, (c) and (d) are stationary. (e) State period of all
runs with $\Tpop=24$ as a function of $L$, for $\TG=8$ (circles) and $\TG=24$
(squares); symbol area grows with the number of runs. The gray line is
Eq.~\eqref{eq:t24}, and the points at $\Tst=24$ are the stationary members.
(f) Percentage of runs ending in a stationary or drifting member of the family,
and in an attractor with $\Tpop=80$.}
\label{fig:t24}
\end{figure*}

\begin{table}[t]
\caption{Members of the $\Tpop=24$ family at $\alpha=1/3$, $r=1$ (selection; all runs
pooled). $h$ is the element of $\Gamma$ with $s_{t+\TG}=h\,s_t$ (the simplest one when
several exist; $\tau_v$ is a translation by $v$, $\sigma$ the color rotation) and $m$ the
number of applications of $h$ that return the configuration. In every member of the family
$\Tst=\TG\,m$ (Proposition~\ref{prop:factor}); $\Tpop=24$ and $\TG\in\{8,24\}$ do not
depend on $L$, whereas $m$, and hence $\Tst$, is set by the order of the translation on
the torus.}
\label{tab:t24}
\footnotesize\setlength{\tabcolsep}{3pt}
\begin{ruledtabular}
\begin{tabular}{rlrlrl}
$L$ & $(\Tpop,\Tst,\Tcol,\Tgeo,\TG)$ & Runs & $h$ & $m$ & $\Tst$\\
\colrule
9 & $(24,72,72,24,8)$ & 4545 & $\sigma\circ\tau_{(0,2)}$ & 9 & 72\\
10 & $(24,120,40,24,8)$ & 83 & $\sigma\circ\tau_{(0,2)}$ & 15 & 120\\
11 & $(24,264,88,24,8)$ & 143 & $\sigma\circ\tau_{(1,1)}$ & 33 & 264\\
11 & $(24,264,264,24,24)$ & 164 & $\tau_{(3,3)}$ & 11 & 264\\
12 & $(24,96,96,24,8)$ & 2087 & $\sigma\circ\tau_{(0,1)}$ & 12 & 96\\
12 & $(24,96,96,24,24)$ & 284 & $\tau_{(3,3)}$ & 4 & 96\\
13 & $(24,312,104,24,8)$ & 1020 & $\sigma\circ\tau_{(0,1)}$ & 39 & 312\\
13 & $(24,624,624,24,24)$ & 93 & $\text{refl.}\circ\tau_{(-2,-3)}$ & 26 & 624\\
14 & $(24,24,8,24,8)$ & 477 & $\sigma$ & 3 & 24\\
15 & $(24,360,360,24,24)$ & 2 & $\tau_{(1,0)}$ & 15 & 360\\
18 & $(24,216,216,24,24)$ & 19 & $\tau_{(0,2)}$ & 9 & 216\\
24 & $(24,96,96,24,8)$ & 13 & $\sigma\circ\tau_{(2,0)}$ & 12 & 96\\
30 & $(24,240,240,24,8)$ & 63 & $\sigma\circ\tau_{(1,1)}$ & 30 & 240\\
30 & $(24,240,240,24,24)$ & 43 & $\tau_{(3,3)}$ & 10 & 240\\
30 & $(24,24,8,24,8)$ & 7458 & $\sigma$ & 3 & 24\\
30 & $(24,24,24,24,24)$ & 239 & identity & 1 & 24\\
\end{tabular}
\end{ruledtabular}
\end{table}

\subsection{Larger and rectangular tori}\label{sec:tests}

The interpretation of Sec.~\ref{sec:t24} has two consequences that can be checked
on lattices outside the range $L\le30$ on which it is based. First,
Eq.~\eqref{eq:t24} should give the state period of the drifting members at larger
sizes. Second, a change of the shape of the domain that keeps the local rule
should change $m$, and with it $\Tst$, while $\Tpop$, $\TG$ and the type of $h$
stay the same. We checked both on square tori with $L=31,\dots,40$ and on a
rectangular $9\times12$ torus (Table~\ref{tab:tests}).

\paragraph{Larger square tori.} For $L=31,\dots,40$, Eq.~\eqref{eq:t24} gives
$\Tst=744$, 768, 264, 816, 840, 288, 888, 912, 312 and 960. We ran the simulation
code with $10^4$ random initial conditions per size
(\prelimL{$10^5$} runs in total). All the \prelimL{482} runs that end in a drifting member of the
family have exactly this state period, both those with $\TG=8$ and
$h=\sigma\circ\tau_v$ and those with $\TG=24$ and $h=\tau_{(3,3)}$; there is
no exception of the kind found at $L=10$, 13, 15, 18 and 24. The stationary
members remain the majority, with \prelimL{3--5\%} of them in the variant
$(24,24,24,24,24)$ as for $L\le30$. The trends of Sec.~\ref{sec:size} continue:
fixed points fall from \prelimL{13.8\%} of the runs at $L=31$ to \prelimL{4.7\%} at
$L=40$, the $\Tpop=24$ family grows from \prelimL{39\%} to \prelimL{55\%}, and
attractors with $\Tpop=80$ take a nearly constant \prelimL{29--33\%}.

\paragraph{A rectangular torus.} On an $L_1\times L_2$ torus with $L_1\neq L_2$
the rotations by $\pm90^\circ$ and the diagonal reflections no longer map the
lattice onto itself, and the symmetry group becomes
$\Gamma=(\Z_{L_1}\times\Z_{L_2})\rtimes(\Z_2\times\Z_2)\times C_3$, with the
rotation by $180^\circ$ and the two axial reflections as point group. The
definitions of Sec.~\ref{sec:periods} and Proposition~\ref{prop:factor} apply
unchanged, with the order of a translation now
$\mathrm{lcm}\bigl(L_1/\gcd(L_1,v_1),\,L_2/\gcd(L_2,v_2)\bigr)$. We chose
$9\times12$, which combines the side of the smallest torus that hosts the
family with the side of a torus where it drifts by one site, and adapted the
simulation code accordingly (the geometric periods are searched over the four
point symmetries and the $9\times12$ translations; for $L_1=L_2$ the modified
code gives output identical to the original). We made \prelim{$10^5$} runs and
re-simulated independently up to three runs of each oscillating tuple
(\prelim{43} runs), with full agreement in all five periods.

The rectangle is markedly more oscillatory than square tori of similar area:
\prelim{82.9\%} of the runs end in a fixed point and \prelim{17.1\%} oscillate,
against \prelim{94.1\%} fixed points at $L=10$ ($N=100$) and \prelim{88.3\%} at
$L=11$ ($N=121$). The \prelim{16} oscillating tuples are listed in
Table~\ref{tab:tests}(b). Nine of them also occur on square tori, seven at $L=9$:
the traveling waves along the short side, the long cycles $(21,189,189,21,21)$,
$(23,69,69,23,23)$, $(18,108,108,18,6)$ and $(14,14,14,14,14)$, and the
color-rotating cycle $(18,18,6,18,6)$. Their generators act only along the side of
length 9 (translations $\tau_{(v_1,0)}$, the reflection of the rows, or no
translation at all), so they do not see the second side and keep their $L=9$ tuple.

The $\Tpop=24$ family appears only as $(24,48,48,24,8)$, a tuple found on no
square torus. Its generator is $h=\sigma\circ\tau_{(0,\pm2)}$: the pattern drifts
by two sites every 8 steps, as at $L=9$ and $L=10$, but along the side of length
12. Then $m=\mathrm{lcm}(3,12/\gcd(12,2))=6$ and $\Tst=48$, against $m=9$ and
$\Tst=72$ at $L=9$, while $\Tpop=24$, $\TG=8$ and the type of $h$ are those of
the whole family. A drift along the short side would give $\Tst=72$, but it was
not observed in \prelim{909} runs. The second most frequent oscillating tuple,
$(21,84,84,21,7)$, is the analogous case for $\Tpop=21$: it shares $\Tpop$, $\TG$
and a generator of the form $\sigma\circ\tau_v$ with the $L=9$ cycle
$(21,63,63,21,7)$, but drifts along the long side, so that $m=12$ instead of 9.
The remaining new tuples, with $\Tpop=25$, 20, 42 and 36, have no counterpart in
our square-torus data. The rectangle therefore changes the lattice-induced part
of the tuple and leaves the intrinsic part untouched, which is the separation
of Sec.~\ref{sec:t24}, here obtained by changing the shape of the domain
instead of its size.

\begin{table*}[t]
\caption{Larger and rectangular tori ($\alpha=1/3$, $r=1$, $P_x=\mathrm{round}(N/3)$).
(a) Square tori with $L=31,\dots,40$: percentage of runs ending in a
fixed point (absorbing or frozen), in the $\Tpop=24$ family and in an attractor with
$\Tpop=80$; number of runs ending in a drifting member of the $\Tpop=24$ family, state period
given by Eq.~\eqref{eq:t24}, and number of drifting runs with
that state period. (b) The rectangular $9\times12$ torus: every oscillating tuple, with the
generator $h$ of Proposition~\ref{prop:factor} ($\rho$ a reflection; translations in units of
rows and columns, the second component along the side of length 12), the factor
$m=\Tst/\TG$, and the square tori on which the same tuple occurs. Both panels show that only the lattice-induced part of the tuple ($m$ and $\Tst$) changes with the size and shape of the torus, while $\Tpop$, $\TG$ and the type of $h$ do not.}
\label{tab:tests}
\begin{ruledtabular}
\begin{tabular}{rrrrrrrr}
\multicolumn{8}{l}{(a) Square tori, $L=31$--40}\\
$L$ & Runs & Fixed (\%) & $\Tpop=24$ (\%) & $\Tpop=80$ (\%) & Drifting & Predicted $\Tst$ & Agree\\
\hline
31 & 10\,000 & 13.8 & 39.1 & 33.3 & 48 & 744 & 48\\
32 & 10\,000 & 12.4 & 38.0 & 33.3 & 50 & 768 & 50\\
33 & 10\,000 & 11.2 & 41.3 & 33.3 & 44 & 264 & 44\\
34 & 10\,000 & 10.2 & 43.0 & 31.8 & 49 & 816 & 49\\
35 & 10\,000 & 9.1 & 42.4 & 32.5 & 51 & 840 & 51\\
36 & 10\,000 & 7.4 & 46.5 & 32.0 & 48 & 288 & 48\\
37 & 10\,000 & 6.6 & 51.5 & 30.7 & 57 & 888 & 57\\
38 & 10\,000 & 5.9 & 51.6 & 30.4 & 44 & 912 & 44\\
39 & 10\,000 & 5.4 & 55.1 & 29.4 & 47 & 312 & 47\\
40 & 10\,000 & 4.7 & 54.2 & 29.7 & 44 & 960 & 44\\
\end{tabular}
\end{ruledtabular}
\begin{ruledtabular}
\begin{tabular}{lrlrl}
\multicolumn{5}{l}{(b) Rectangular torus $9\times12$ (\prelim{$10^5$} runs)}\\
Tuple & Runs & $h$ & $m$ & Occurrence on square tori\\
\hline
$(1,12,12,1,1)$ & 8\,992 & $\tau_{(0,1)}$ & 12 & same tuple at $L=12$\\
$(21,84,84,21,7)$ & 3\,408 & $\sigma\circ \tau_{(0,-1)}$ & 12 & new; cf.\ $(21,63,63,21,7)$, $L=9$\\
$(24,48,48,24,8)$ & 909 & $\sigma\circ \tau_{(0,-2)}$ & 6 & new; cf.\ $(24,72,72,24,8)$, $L=9$\\
$(25,75,75,25,25)$ & 813 & $\tau_{(3,0)}$ & 3 & new\\
$(1,12,4,1,1)$ & 801 & $\tau_{(0,-1)}$ & 12 & same tuple at $L=12$\\
$(20,240,240,20,20)$ & 709 & $\rho\circ \tau_{(-4,-1)}$ & 12 & new\\
$(18,18,6,18,6)$ & 627 & $\sigma$ & 3 & same tuple at $L=9, 10, 12, 14$\\
$(42,168,168,42,14)$ & 621 & $\sigma^2\circ \tau_{(0,-1)}$ & 12 & new\\
$(1,9,9,1,1)$ & 162 & $\tau_{(1,0)}$ & 9 & same tuple at $L=9$\\
$(1,9,3,1,1)$ & 25 & $\tau_{(1,0)}$ & 9 & same tuple at $L=9$\\
$(21,189,189,21,21)$ & 16 & $\tau_{(-4,0)}$ & 9 & same tuple at $L=9$\\
$(18,108,108,18,6)$ & 7 & $\sigma\circ \rho\circ \tau_{(1,0)}$ & 18 & same tuple at $L=9$\\
$(23,69,69,23,23)$ & 4 & $\tau_{(-3,0)}$ & 3 & same tuple at $L=9$\\
$(36,36,12,36,12)$ & 2 & $\sigma^2$ & 3 & new\\
$(14,14,14,14,14)$ & 1 & identity & 1 & same tuple at $L=9$\\
$(9,18,6,9,3)$ & 1 & $\sigma^2\circ \rho\circ \tau_{(-4,6)}$ & 6 & same tuple at $L=10$\\
\end{tabular}
\end{ruledtabular}
\end{table*}

\subsection{Periods of individual sites}\label{sec:sites}

The site periods $T_i$ add a local level to the tuple. We computed them for up to
ten final configurations of every oscillating tuple at $\alpha=1/3$, $r=1$, for
$L=5,\dots,40$ and for the $9\times12$ torus (about \prelim{7000} configurations,
re-simulated over one state period). Three results hold without exception. First,
$\mathrm{lcm}_i\,T_i=\Tst$ in every configuration. Second, no site is frozen on
any oscillating attractor: every site completes at least one turn
$\mathrm{X}\to\mathrm{Y}\to\mathrm{Z}\to\mathrm{X}$ per state period. Third, the
oscillation is collective on the small lattices: at $L=8$ and $L=9$ all the sites
of an attractor share the period $\Tst$, except in $(6,12,12,6,6)$ and
$(18,108,108,18,6)$, where a few sites have half of it (Table~\ref{tab:sites} of the Supplemental Material \cite{SM}).
The same holds on the $9\times12$ torus, with one exception, $(36,36,12,36,12)$.

The site periods separate the two spiral families of large lattices. In the
stationary $\Tpop=24$ attractors every site has $T_i=24$ and rotates at
$\bar\omega=1/8$; in the variant without color symmetry some sites near the
cores make two turns instead of one, and at $L=13$ a single core site has
$T_i=6$. In the $\Tpop=80$ attractors every site has $T_i=80$ and makes three or
four turns, with $\bar\omega\approx\prelim{0.115}$, slightly slower than the
$\Tpop=24$ spirals: they are a different local rhythm, not a superposition of
the first one. Finally, the site periods show that many of the long cycles of
large lattices are composite. For $L\ge14$, in \prelim{58\%} of the configurations
of the long cycles with $\Tpop\notin\{1,24,80\}$ (between \prelim{34\%} and
\prelim{74\%} depending on $L$) part of the lattice has $T_i=24$, and in these
configurations it is most of the lattice: \prelim{72\%} of the sites on average
(median \prelim{80\%}), while the remaining sites have the long period $\Tst$.
These attractors therefore consist of a large region locked to the rhythm of the
$\Tpop=24$ spirals and a smaller region whose motion sets the period of the whole
configuration. The
pattern of site periods thus localizes, inside a single attractor, the part of
the lattice responsible for a long state period, something that none of the
global periods can do.

The site periods also show that the tuple is not a complete classification of
the attractors. We call the multiset of pairs $(T_i,\nu_i)$ over the sites the
\emph{site-period spectrum}; since sites related by a symmetry of $\Gamma$ have the
same period and the same number of turns, it is an invariant of the attractor up
to symmetry, and it refines the tuple without adding a new global period (its
only scalar summary, $\mathrm{lcm}_i T_i$, is $\Tst$ itself). On the lattices of
the main study the refinement is almost empty: at $L=8$ and $L=9$ only one tuple,
$(6,12,12,6,6)$ at $L=8$, contains configurations with different spectra (all 64
sites with $T_i=12$, or four of them with $T_i=6$). On large lattices it is
substantial: for $L\ge14$, \prelim{419} of the \prelim{1096} pairs $(L,\text{tuple})$ contain
configurations with different spectra (between \prelim{17\%} and \prelim{60\%} of the
tuples depending on $L$), mostly long cycles such as $(168,168,56,168,56)$ or
$(72,72,72,72,72)$, whose sampled configurations often all have different spectra.
Since at most ten configurations per tuple were examined, these numbers are lower
bounds. On small lattices the tuple identifies each long cycle up to symmetry
(with the few exceptions of Table~\ref{tab:long} of the Supplemental Material), while on large lattices it
identifies a family of attractors with the same periods, which the site-period
spectrum resolves further. The spectrum does not separate everything either: all
traveling waves of a given size have $T_i=L$ and one turn per site, whatever their
band widths.

\subsection{Dependence on the other parameters}\label{sec:params}

The threshold, radius and composition sweeps confirm the picture expected from
the analysis of Fisch, Gravner and Griffeath (Sec.~\ref{sec:model}); the
details are given in Sec.~\ref{sec:S-params} of the Supplemental Material
\cite{SM}. For $\alpha\le1/4$ ($\theta\le2$) almost every run ends in a uniform
rotation, for $\alpha\ge1/2$ ($\theta\ge4$) more than \prelim{99.5\%} of the runs
freeze and no oscillating attractor is observed, and $\alpha=3/8$ gives the same
frequencies and tuples as $\alpha=1/3$, so that the whole repertoire of
attractors is confined to $\theta=3$. Increasing the radius shifts the balance
from takeover to frozen coexistence and leaves only traveling waves, with jagged
fronts, among the oscillating attractors. Oscillating attractors appear only
close to the balanced composition and are most frequent at
$P_x=\mathrm{round}(N/3)$: coexistence of the three states requires that none
of them is initially in a clear minority.

\section{Discussion}
\label{sec:discussion}

The main message of this work is that the period of a periodic attractor is
not one number but a small family of numbers, one for each level of description,
and that the symmetry group of the dynamics organizes this family in two ways.
It imposes exact divisibility relations, so that the ratios between the periods
identify the space--color symmetry that generates the cycle; and, through the
factorization $\Tst=\TG\,m$ of Proposition~\ref{prop:factor}, it separates what
belongs to the pattern ($\TG$, $\Tpop$ and the type of symmetry) from what is
induced by the finite domain ($m$, and with it $\Tst$, $\Tcol$ and $\Tgeo$).

\paragraph{What each period reveals} The population period, the only one
accessible from aggregated data, is blind to rigid motion: traveling waves, the
majority of the oscillating attractors on small lattices, have $\Tpop=1$, and in
the drifting cycles $\Tpop=\Tgeo$ underestimates the microscopic period by the
order of the translation (a factor between 2 and 9 at $L\le9$, and up to 40 for the drifting
$\Tpop=24$ cycles of the largest lattices). Since
$\Tst/\Tcol\in\{1,3\}$, $\Tcol$ is a binary indicator of color-rotation
symmetry; $\Tst/\Tgeo$ is the order of the geometric transformation, which
explains why long cycles have periods that are multiples of 3 and 9 at $L=9$ and
of 4 at $L=8$; and $\TG$ reveals symmetries that neither $\Tcol$ nor $\Tgeo$
detect, such as the color-rotation-plus-translation motif of 8 steps in the
72-step cycle. In the catalogue of $L\le9$ the pair $(\Tcol,\Tgeo)$ already
separates all the observed tuples, while on larger lattices $\TG$ is also needed,
since drifting cycles with $\TG=8$ and $\TG=24$ can share the other four periods
(Sec.~\ref{sec:S-further} of the Supplemental Material \cite{SM}). The
divisibility relations of Corollary~\ref{cor:lattice} are also a useful
consistency check for simulation codes: any violation signals an error in the
period detection.

\paragraph{Intrinsic and lattice-induced periods} The $\Tpop=24$ family
(Sec.~\ref{sec:t24}) is, in our view, the most important physical result of
this work. One object, a pair of counter-rotating spiral defects, appears at
$L=9$ and is found at every larger size up to $L=40$. Its population period,
its symmetry period, the type of its generating symmetry and its drift velocity
do not depend on $L$, whereas its state period takes values from 24 to 960, all
given by Eq.~\eqref{eq:order}; on a $9\times12$ torus the family drifts along
the long side with $\Tst=48$, a value found on no square torus. The tuple thus
provides an operational way to separate the intrinsic properties of a pattern
from those induced by the domain, a distinction usually made by inspecting
snapshots. The same separation should apply to any equivariant dynamics on a
finite domain with periodic boundary conditions: a pattern with an intrinsic
period and a velocity commensurate with the lattice acquires a state period
equal to its intrinsic one times the order of its translation, and comparing the
tuple across sizes tests whether two attractors are the same object. The site
periods (Sec.~\ref{sec:sites}) complete the picture at the local level: the
stationary spirals rotate as a single phase field, many long cycles of large
lattices are composite, with most of the lattice locked to the rhythm of the
spirals, and the site-period spectrum separates attractors that share a tuple,
which on large lattices is frequent.

\paragraph{Dynamics and diversity} Cyclic threshold dynamics on a small torus
almost always destroys local diversity: one state takes over or two states
freeze side by side, and the three states coexist only in the 2--5\% of runs at
$L\le9$ that end in an oscillating attractor. In larger lattices the picture is
reversed: oscillating attractors become the majority (\prelim{84\%} of the runs
at $L=30$) and the spiral attractors with $\Tpop=24$ and $\Tpop=80$ overtake
the bands, as the analogy with spiral waves in excitable media suggests
\cite{Fisch1991a,Reichenbach2007}. The symmetry analysis connects these
observations with diversity: the diversity counts depend only on the intrinsic
dynamics, a color symmetry of the cycle guarantees equal average populations,
and the spiral defects guarantee three-state neighborhoods at their cores, so
that on large lattices local diversity is a topological property of the
attractor, carried by a number of sites that does not grow with the area.

\paragraph{Relation to excitable media} The attractors that dominate large lattices are
the discrete counterpart of the spiral waves of excitable and oscillatory media,
from the Greenberg--Hastings automaton \cite{Greenberg1978} to the
Belousov--Zhabotinsky reaction \cite{Winfree1972,Kapral1995}. In those systems,
too, the measured period of a confined spiral mixes its intrinsic rotation with
the effect of the boundaries, and drifting or meandering spirals are relative
periodic orbits of the Euclidean group \cite{Barkley1994}. The symmetry-resolved
periods offer a simple way to separate the two in discrete simulations: $\TG$
and the type of drift characterize the spiral, whereas $\Tst$ also depends on how
the drift closes on the domain. Since the construction uses only the
equivariance of the dynamics, it applies without change to other cellular
automata and coupled-map lattices with the same lattice and color symmetries,
and, with a continuous group, to spatiotemporal data from simulations of
reaction--diffusion models; whether it is also useful for experimental image
sequences, where symmetries hold only approximately, is an open question.

\paragraph{Relation to previous work} Earlier results on the periodicity of
cyclic automata bound the period of a site or of the whole system in terms of
the number of colors or of the cycle structure of the underlying graph
\cite{Fisch1991b,Matamala1997,Matamala2004,Doba2018}. The tuple of periods is
complementary: rather than bounding one period, it decomposes it, through the
symmetry group of the model, into a part due to the motion of the pattern
($\Tst/\Tgeo$), a part due to the cyclic exchange of the species
($\Tst/\Tcol$) and an elementary motif ($\TG$), and connects it with what is
observable in aggregated data ($\Tpop$).

\paragraph{Limitations and outlook} The catalogue of tuples is complete for
practical purposes at $L=8$ and $L=9$, but not for $L\ge16$; for $L\ge17$
about \prelim{0.01--0.07\%} of new runs are expected to end in a tuple not yet
observed (Sec.~\ref{sec:S-stats} of the Supplemental Material). Open questions
are the relation between the spiral pairs of the $\Tpop=24$ and $\Tpop=80$
attractors and the spirals of the cyclic cellular automaton on infinite lattices
\cite{Fisch1991b}, the minimum size needed to host them, the number of distinct
attractors per tuple on large lattices, and the extension of the analysis to
asymmetric invasion rates, for which the color group $C_3$ is no longer a
symmetry, and to irregular networks, where the geometric group is replaced by the
automorphism group of the graph. Finally, since every run records its initial condition together with the
tuple of its attractor, the data set allows one to pose the prediction of the
asymptotic behavior from the initial configuration as a supervised learning
problem, which we leave for future work.

The main outcome of this work is therefore not the catalogue of observed periods
but the recognition that periodicity itself is a layered concept in cyclic
cellular automata. Once the symmetries are taken into account, several
apparently unrelated time scales become manifestations of a common structure: a
short intrinsic motif, repeated with a drift whose order is set by the lattice.

\appendix
\section{Simulation protocol and verification}
\label{app:study}

This appendix summarizes how the periods were measured and checked. The pooling
of the sweeps, the estimates of unobserved tuples, further results and the
simulations of larger lattices are given in the Supplemental Material
\cite{SM}.

\subsection{Simulation protocol}

Each run of the simulation code (written in R with the igraph library)
consists of the following steps.
\begin{enumerate}
\item \emph{Initial condition.} A string of $N$ symbols with exactly $P_x$ X,
$P_x$ Y and $N-2P_x$ Z is shuffled uniformly at random. This string is stored
in the output (column \texttt{ci}), which makes every run reproducible.
\item \emph{Update.} The configuration is updated synchronously with
Eq.~\eqref{eq:rule}. The neighborhood is built as the Moore graph of radius $r$
on the torus with multiple edges removed, so that each site has
$\min\{(2r+1)^2,N\}-1$ distinct neighbors; for $L=9$, $r=4$ the neighborhood is
the whole torus.
\item \emph{Period detection.} After every step the code appends to ten time
series their current value: the configuration, the three populations, the three
diversity counts, the configuration without color, the vector of local
diversities and the vector of changes. For each series it keeps, for every lag
$p$, the length $M_p$ of the current run of coincidences $x_k=x_{k-p}$. The
detected period is the smallest $p\le n/3$ for which the periodic tail contains
at least $2p+1$ values, and the transient is the index at which that tail
starts. This is equivalent to an exhaustive search for the shortest periodic
tail, but costs $O(n)$ operations per step.
\item \emph{Stopping rule.} A run stops when (a) the lattice is uniform or
frozen, or (b) the configuration, the populations and the diversity counts have
a detected period, in both cases only after the population and diversity
periods have been detected for at least ten consecutive steps. The population
period of each state is the largest period detected during those steps, and
$\Tpop$ is the largest of the three (which equals their least common multiple in
all our attractors).
\item \emph{Symmetry periods.} Once the run has stopped, $\Tgeo$ is obtained by
comparing the last configuration with the $\Tst$ previous ones under the $8L^2$
elements of $G$, and $\TG$ by doing the same under the $24L^2$ elements of
$\Gamma$; the images of a configuration under all the lattice transformations
are generated at once from precomputed permutations of the site indices. $\Tcol$
is the period of the series of configurations without color.
\item \emph{Output.} One line per run is appended to a results file, with the parameters, the initial and final
configurations, the final populations and diversity counts, the transient
length, all the periods and the group elements found at $\Tgeo$ and $\TG$.
\end{enumerate}
The requirement of $2p+1$ values and ten confirmations makes the code run a few
steps beyond the moment at which the attractor is first reached: in the runs
checked below the median overhead is \prelim{10} steps and the maximum
\prelim{367}, for the longest cycles.

\subsection{Data set and pooling}

The runs were organized in series of up to $10^4$ realizations per parameter
value, about $10^5$ per value in total (Table~\ref{tab:sweeps}), all with
independent initial conditions; the data set contains about \prelim{$10^7$}
runs, \prelim{$3.0\times10^6$} of them at the reference point. For a subset of
\prelim{$2.5\times10^5$} runs, $\TG$ was obtained by re-simulating the final
configuration. Runs at the reference point produced by different sweeps are
statistically equivalent ($\chi^2$ test of homogeneity, $p=\prelim{0.32}$ at
$L=8$ and $p=\prelim{0.23}$ at $L=9$; Table~\ref{tab:series} of the
Supplemental Material) and are pooled.

\subsection{Independent implementation and verification}

An independent implementation in Python/NumPy was written for the checks and for
the larger lattices. It does not detect periods from time series: it stores
every configuration of a trajectory in a hash table and stops at the first
repetition, which gives the transient $\tau$ and $\Tst$ exactly. The remaining
periods are then computed from one full cycle: $\Tpop$ as the least common
multiple of the periods of the three populations, $\Tcol$ both with the cyclic
color shifts $C_3$ and with the first-appearance relabelling of the R code, and
$\Tgeo$ and $\TG$ by testing all the elements of $G$ and $\Gamma$.

Three checks were made. (i) \emph{Complete runs.} We re-simulated
\prelim{$2.4\times10^4$} runs from their recorded initial conditions: all the
oscillating runs at the reference point in a set of \prelim{$1.2\times10^5$}
runs (\prelim{2498} at $L=8$ and \prelim{1978} at $L=9$) and \prelim{$10^4$}
randomly chosen non-oscillating runs per size. The class of the attractor, the
five periods and the transient length agree in \prelim{100\%} of them. The
\prelim{570} distinct oscillating cycles reached in this sample give the same
$\Tcol$ with $C_3$ and with the first-appearance relabelling, as anticipated in
the Remark of Sec.~\ref{sec:periods}. (ii) \emph{Final states.} Re-simulating
the final configuration of \prelim{$9.7\times10^4$} further runs reproduced
$\Tst$, $\Tpop$, $\Tcol$ and $\Tgeo$ in all cases, and the recorded $\TG$ agreed
with the re-simulated one in all the \prelim{1463} oscillating attractors
tested. (iii) \emph{Divisibility.} The relations (a)--(c) of
Corollary~\ref{cor:lattice} hold in all the \prelim{$1.2\times10^6$} oscillating runs of
the data set. The relation $\Tdiv\mid\TG$ fails in \prelim{21} runs, all in the
$L$ sweep at $L\ge15$. Nineteen of them are on the same attractor,
$(192,192,64,192,64)$, reached at $L=15$ (fifteen runs), 16 (three) and 26: the series of the
diversity count $n_3$ on this cycle contains a stretch that repeats with period
14 for more than 28 steps, which the detector took as the period, whereas the
exact period of $n_3$, obtained by re-simulation, is 64. The other two, with
recorded $n_3$ periods of 22 (on $(120,120,120,120,120)$ at $L=20$) and 31 (on
a cycle with $\Tpop=1632$ at $L=28$), have the same origin: on long cycles the
diversity series can contain a quasi-periodic stretch long enough to pass the
detection criterion. These are the only measurement errors that we found, and
they affect only the diversity periods, never the five periods of the tuple.

The median transients are \prelim{15}, \prelim{10} and \prelim{9} steps for
absorbing, frozen and oscillating runs at $L=8$, and \prelim{20}, \prelim{13}
and \prelim{13} at $L=9$ (Table~\ref{tab:census} gives the medians over all
runs); the longest transient in the re-simulated sample is \prelim{70} steps.
These short transients are what makes exact sampling possible.

\section*{Data availability}
The simulation code (written in R) and the Python scripts used for the
verification, the tables and the figures are available from the authors upon
reasonable request.

\bibliographystyle{apsrev4-2}
\bibliography{refs}

\clearpage
\onecolumngrid
\makeatletter\renewcommand{\fps@table}{!htbp}\renewcommand{\fps@figure}{!htbp}\makeatother
\setcounter{figure}{0}\setcounter{table}{0}\setcounter{equation}{0}
\renewcommand{\thetable}{S\arabic{table}}
\renewcommand{\thefigure}{S\arabic{figure}}
\renewcommand{\theequation}{S\arabic{equation}}
\newcounter{Ssec}\renewcommand{\theSsec}{S\arabic{Ssec}}
\newcommand{\Ssection}[2]{\refstepcounter{Ssec}\label{#2}\section*{\theSsec.\quad #1}}
\begin{center}{\large\bfseries Supplemental Material}\end{center}
\begin{quotation}\small
This Supplemental Material contains the table of long cycles and their
space--time diagrams, the complete catalogue of oscillating tuples, the periods
of individual sites, the dependence on the threshold, radius and initial
composition, the pooling of the parameter sweeps, the estimates of the
unobserved part of the catalogue of tuples with the sample statistics for every
lattice size, further results on the relation between aggregated and microscopic
periods, and the comparison with the independent simulations of larger lattices. Equations, figures, tables and
sections without the prefix S refer to the main text.
\end{quotation}

\Ssection{Long cycles and space--time diagrams}{sec:S-long}

Table~\ref{tab:long} lists the long cycles found at $L\le9$ (Sec.~\ref{sec:long}), and Fig.~\ref{fig:spacetime} shows space--time diagrams of a traveling wave and of two drifting long cycles.

\begin{table*}[!htbp]
\centering
\caption{Long cycles at $\alpha=1/3$, $r=1$ (all runs pooled).
``Runs'' counts all runs; ``Distinct'' is the number of different cycles up to
$\Gamma$ and time shifts; ``drift'' the transformation after $\Tgeo$ steps;
populations give the range of $n_c(t)$ over the cycle; $\tau_{1/2}$ is the
median transient.}
\label{tab:long}
\small
\begin{tabular}{rrrrrlrr}
\toprule
$L$ & $\Tst$ & Runs & Distinct & $\Tgeo$ & Drift & Populations & $\tau_{1/2}$\\
\midrule
6 & 8 & 13 & 4 & 4 & $(0,3)$ & 6--17 & 4\\
6 & 9 & 1 & 1 & 9 & none (color rotation) & 10--14 & 5\\
6 & 12 & 34 & 2 & 6 & $(3,3)$ + color rotation & 10--14 & 7.5\\
6 & 60 & 12 & 1 & 15 & rotation + translation & 11--13 & 5\\
\addlinespace
7 & 9 & 2 & 1 & 9 & none (color rotation) & 14--21 & 6.5\\
7 & 12 & 1 & 1 & 3 & rotation + $(3,2)$ & 14--19 & 8\\
7 & 12 & 2 & 1 & 12 & none (color rotation) & 15--18 & 4.5\\
7 & 60 & 4 & 1 & 15 & rotation + translation & 15--18 & 7\\
7 & 147 & 1 & 1 & 21 & $(0,2)$ + color rotation & 7--30 & 8\\
\addlinespace
8 & 12 & 11 & 2 & 6 & $(0,4)$ + color rotation & 16--24 & 11\\
8 & 12 & 32 & 5 & 6 & 4 sites & 11--35 & 16\\
8 & 16 & 239 & 2 & 4 & rotation + $(4,4)$ & 16--24 & 15\\
8 & 88 & 95 & 1 & 22 & 2 sites & 8--40 & 13\\
\addlinespace
9 & 14 & 20 & 1 & 14 & none & 13--36 & 28.5\\
9 & 18 & 1715 & 2 & 18 & none (color rotation) & 20--37 & 36\\
9 & 54 & 39 & 1 & 18 & 3 sites & 14--45 & 12\\
9 & 63 & 19 & 1 & 21 & $(3,-3)$ & 8--53 & 17\\
9 & 69 & 582 & 1 & 23 & 3 sites & 10--47 & 20\\
9 & 72 & 4545 & 1 & 24 & 3 sites & 10--56 & 33\\
9 & 108 & 189 & 1 & 18 & reflection + $(3,3)$ & 14--45 & 17\\
9 & 189 & 1271 & 1 & 21 & 4 sites & 12--53 & 26\\
\bottomrule
\end{tabular}
\end{table*}

\begin{figure*}[!htbp]
\centering
\includegraphics[width=\textwidth]{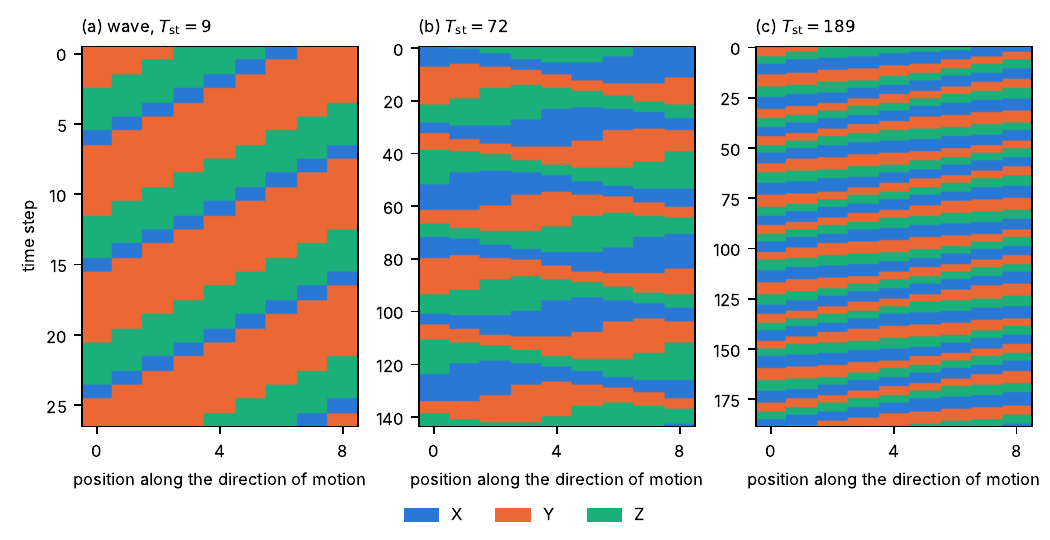}
\caption{Space--time diagrams of one line of sites taken along the direction
of motion. (a) Traveling wave ($L=9$, $r=1$, bands 4-3-2): the fronts advance
one site per step and the pattern returns after $\Tst=L=9$ steps.
(b) Long cycle with $\Tst=72$: the pattern drifts three sites every 24 steps,
i.e.\ $1/8$ site per step. (c) Long cycle with $\Tst=189$: four sites every 21
steps. In (b) and (c) the dominant state changes cyclically
X$\to$Y$\to$Z while the pattern drifts slowly.}
\label{fig:spacetime}
\end{figure*}

\Ssection{Catalogue of attractors and site periods}{sec:S-cat}

Table~\ref{tab:catalogue} lists every oscillating tuple found at $\alpha=1/3$ with $L\le9$ (all radii), with the generator $h$ of Proposition~\ref{prop:factor}; Table~\ref{tab:sites} gives the periods of individual sites discussed in Sec.~\ref{sec:sites}.

\begin{table*}[!htbp]
\caption{Catalogue of the oscillating attractors found at $\alpha=1/3$ from
random initial conditions with $P_x=\mathrm{round}(N/3)$ (all runs pooled). The next-to-last column gives an element of $\Gamma$ that maps
$s_t$ onto $s_{t+\TG}$ (translations in lattice units, given by their shortest
representative up to sign; $\sigma$ is the color rotation
X$\to$Y$\to$Z$\to$X; ``rot.'' a rotation of the square). ``Runs'' is the
number of realizations ending in that tuple; the corresponding fractions of all
runs are given for $L=8,9$ in Table~\ref{tab:basins}. The case $L=9$, $r=4$, in which
the neighborhood covers the whole torus and every run ends in a uniform
rotation, is omitted. Every tuple satisfies $\Tst=\TG\,m$ with the listed generator, and the ratios between its periods identify the type of drift.}
\label{tab:catalogue}
\footnotesize
\begin{ruledtabular}
\begin{tabular}{rrlrrrrrlr}
$L$ & $r$ & Class & $\Tpop$ & $\Tst$ & $\Tcol$ & $\Tgeo$ & $\TG$ & Generator at $\TG$ & Runs\\
\colrule
5 & 1 & traveling wave & 1 & 5 & 5 & 1 & 1 & transl. $(1,0)$ & \prelim{2765} \\[2pt]
6 & 1 & traveling wave & 1 & 6 & 6 & 1 & 1 & transl. $(1,0)$ & \prelim{2730} \\
6 & 1 & wave (equal bands) & 1 & 6 & 2 & 1 & 1 & transl. $(1,0)$ & \prelim{1325} \\
6 & 1 & color-rotating & 6 & 12 & 4 & 6 & 2 & $\sigma^2\circ$transl. $(3,3)$ & \prelim{34} \\
6 & 1 & color-rotating & 9 & 9 & 3 & 9 & 3 & $\sigma$ & \prelim{1} \\
6 & 1 & color-rotating & 15 & 60 & 20 & 15 & 5 & $\sigma^2\circ$rot.\,$90^\circ\circ$transl. $(2,1)$ & \prelim{12} \\
6 & 1 & geometric & 4 & 8 & 8 & 4 & 4 & transl. $(0,3)$ & \prelim{13} \\[2pt]
7 & 1 & traveling wave & 1 & 7 & 7 & 1 & 1 & transl. $(1,0)$ & \prelim{3751} \\
7 & 1 & color-rotating & 9 & 9 & 3 & 9 & 3 & $\sigma$ & \prelim{2} \\
7 & 1 & color-rotating & 12 & 12 & 4 & 12 & 4 & $\sigma$ & \prelim{2} \\
7 & 1 & color-rotating & 15 & 60 & 20 & 15 & 5 & $\sigma^2\circ$rot.\,$90^\circ\circ$transl. $(3,3)$ & \prelim{4} \\
7 & 1 & color-rotating & 21 & 147 & 49 & 21 & 7 & $\sigma\circ$transl. $(0,2)$ & \prelim{1} \\
7 & 1 & geometric & 3 & 12 & 12 & 3 & 3 & rot.\,$90^\circ\circ$transl. $(3,2)$ & \prelim{1} \\[2pt]
8 & 1 & traveling wave & 1 & 8 & 8 & 1 & 1 & transl. $(1,0)$ & \prelim{13\,801} \\
8 & 1 & geometric & 1 & 16 & 16 & 4 & 4 & rot.\,$90^\circ\circ$transl. $(4,4)$ & \prelim{239} \\
8 & 1 & geometric & 22 & 88 & 88 & 22 & 22 & transl. $(0,2)$ & \prelim{95} \\
8 & 1 & geometric & 6 & 12 & 12 & 6 & 6 & transl. $(4,0)$ & \prelim{32} \\
8 & 1 & color-rotating & 6 & 12 & 4 & 6 & 2 & $\sigma^2\circ$transl. $(0,4)$ & \prelim{11} \\
8 & 2 & traveling wave & 1 & 8 & 8 & 1 & 1 & transl. $(1,0)$ & \prelim{2146} \\
8 & 2 & wave (breathing front) & 1 & 8 & 8 & 2 & 2 & transl. $(2,0)$ & \prelim{1} \\[2pt]
9 & 1 & traveling wave & 1 & 9 & 9 & 1 & 1 & transl. $(1,0)$ & \prelim{12\,067} \\
9 & 1 & geometric & 24 & 72 & 72 & 24 & 8 & $\sigma\circ$transl. $(2,0)$ & \prelim{4545} \\
9 & 1 & color-rotating & 18 & 18 & 6 & 18 & 6 & $\sigma$ & \prelim{1715} \\
9 & 1 & wave (equal bands) & 1 & 9 & 3 & 1 & 1 & transl. $(1,0)$ & \prelim{1795} \\
9 & 1 & geometric & 21 & 189 & 189 & 21 & 21 & transl. $(0,4)$ & \prelim{1271} \\
9 & 1 & geometric & 23 & 69 & 69 & 23 & 23 & transl. $(3,0)$ & \prelim{582} \\
9 & 1 & geometric & 18 & 108 & 108 & 18 & 6 & $\sigma\circ$refl.$\circ$transl. $(3,1)$ & \prelim{189} \\
9 & 1 & geometric & 18 & 54 & 54 & 18 & 6 & $\sigma\circ$transl. $(0,1)$ & \prelim{39} \\
9 & 1 & asymmetric & 14 & 14 & 14 & 14 & 14 & identity & \prelim{20} \\
9 & 1 & geometric & 21 & 63 & 63 & 21 & 7 & $\sigma\circ$transl. $(1,-1)$ & \prelim{19} \\
9 & 2 & traveling wave & 1 & 9 & 9 & 1 & 1 & transl. $(1,0)$ & \prelim{3294} \\
9 & 2 & wave (equal bands) & 1 & 9 & 3 & 1 & 1 & transl. $(1,0)$ & \prelim{133} \\
9 & 3 & wave (equal bands) & 1 & 9 & 3 & 1 & 1 & transl. $(1,0)$ & \prelim{79} \\
\end{tabular}
\end{ruledtabular}
\end{table*}

\begin{table}[!htbp]
\caption{Periods of individual sites on the oscillating attractors ($\alpha=1/3$, $r=1$; up to ten
re-simulated final configurations per tuple and size). ``Site periods'' lists the values of
$T_i$ found (with the percentage of sites when there is more than one); ``Turns'' is the number
of complete cycles $\mathrm{X}\to\mathrm{Y}\to\mathrm{Z}\to\mathrm{X}$ made by a site in one
state period, and $\bar\omega$ the mean number of color changes per site and step. No site is
frozen on any oscillating attractor, and $\mathrm{lcm}_i\,T_i=\Tst$ in every configuration.}
\label{tab:sites}
\begin{ruledtabular}
\begin{tabular}{rlllr}
$L$ & Tuple & Site periods & Turns & $\bar\omega$\\
\hline
8 & $(1,16,16,4,4)$ & 16 & 1 & 0.188\\
8 & $(1,8,8,1,1)$ & 8 & 1 & 0.375\\
8 & $(22,88,88,22,22)$ & 88 & 4--7 & 0.213\\
8 & $(6,12,12,6,6)$ & 6 (4\%), 12 (96\%) & 1--2 & 0.263\\
8 & $(6,12,4,6,2)$ & 12 & 1 & 0.250\\
9 & $(1,9,3,1,1)$ & 9 & 1 & 0.333\\
9 & $(1,9,9,1,1)$ & 9 & 1 & 0.333\\
9 & $(14,14,14,14,14)$ & 14 & 1 & 0.214\\
9 & $(18,108,108,18,6)$ & 54 (11\%), 108 (89\%) & 6--8 & 0.210\\
9 & $(18,18,6,18,6)$ & 18 & 1 & 0.167\\
9 & $(18,54,54,18,6)$ & 54 & 3--4 & 0.210\\
9 & $(21,189,189,21,21)$ & 189 & 9--13 & 0.182\\
9 & $(21,63,63,21,7)$ & 63 & 3--4 & 0.180\\
9 & $(23,69,69,23,23)$ & 69 & 3--5 & 0.198\\
9 & $(24,72,72,24,8)$ & 72 & 3--5 & 0.171\\
\hline
$14$--$40$ & $(24,24,8,24,8)$ & 24 & 1 & 0.125\\
$14$--$40$ & $(24,24,24,24,24)$ & 24 & 1--2 & 0.126\\
$19$--$40$ & $(80,80,80,80,80)$ & 80 & 3--4 & 0.116\\
\end{tabular}
\end{ruledtabular}
\end{table}

\Ssection{Dependence on the other parameters}{sec:S-params}

\paragraph{Threshold.} For $r=1$ only $\lceil 8\alpha\rceil$ matters, and the
sweep over $\alpha=m/8$ confirms the thresholds expected from the analysis of
Fisch, Gravner and Griffeath (Sec.~\ref{sec:model}). For
$\alpha\le1/8$ ($\theta\le1$) every run ends in a uniform rotation,
$(\Tpop,3,1,3,1)$, with $\Tpop=1$ only when the three populations are equal; for
$\alpha=1/4$ ($\theta=2$) uniform rotations still take \prelim{99.2\%} of the runs
at $L=8$ and \prelim{99.7\%} at $L=9$, the rest being traveling waves and a few absorbing states.
The value $\alpha=3/8$ defines the same rule as $\alpha=1/3$, and the two
sweeps give the same frequencies within error (at $L=9$:
\prelim{77.8\%}, \prelim{17.2\%} and \prelim{5.0\%} of absorbing, frozen and
oscillating runs for $\alpha=3/8$, against \prelim{78.1\%}, \prelim{17.0\%} and
\prelim{4.9\%} for $\alpha=1/3$), and the same tuples; we therefore pool them in
Tables~\ref{tab:census} and \ref{tab:basins}. For $\alpha\ge1/2$ ($\theta\ge4$) no front can
advance: \prelim{more than 99.5\%} of the runs freeze, for $\alpha\ge5/8$
already in the initial configuration, and no oscillating attractor was observed. The
whole repertoire of attractors is therefore confined to the single threshold
$\theta=3$.

\paragraph{Radius.} Increasing $r$ shifts the balance from takeover to frozen
two-state coexistence (Table~\ref{tab:census}): at $L=9$ the frozen fraction
goes from \prelim{17\%} ($r=1$) to \prelim{39\%} ($r=2$) and \prelim{51\%} ($r=3$),
and the oscillating fraction falls from \prelim{4.9\%} to \prelim{3.0\%} and
\prelim{0.1\%}. For $r\ge2$ all oscillating attractors are traveling waves,
although with jagged fronts. For $L=9$ and $r=4$ the neighborhood covers
the whole torus, every site sees the global composition, and every run ends in a
uniform rotation.

\paragraph{Initial composition.} With $n_\mathrm{X}(0)=n_\mathrm{Y}(0)=P_x$ and
the rest in Z, the outcome depends strongly on $P_x$. For $P_x=0$ the lattice is
uniform, and for $P_x=1,2$ every run freezes (a few isolated X and Y sites cannot
invade the Z sea, and are invaded at most partially). From $P_x=3$ the absorbing fraction grows quickly, up to more than
\prelim{97\%} around $P_x\approx9$--$14$, and then decreases again towards the balanced composition, where
the frozen fraction recovers to about \prelim{17--20\%}. Oscillating attractors
appear only close to the balanced composition: they exceed \prelim{0.1\%} of the
runs at $L=8$ for $P_x\ge18$ and at $L=9$ for $P_x\ge22$ (with $N/3=21.3$ and
$27$, respectively; isolated cases occur from $P_x=15$), and their frequency
grows monotonically up to $P_x=\mathrm{round}(N/3)$ (\prelim{3.2\%} and
\prelim{4.9\%}). Coexistence of the three states in an oscillating attractor
therefore requires that none of them is initially in a clear minority.

\Ssection{Pooling of the sweeps}{sec:S-pool}

The runs were organized in series of up to $10^4$ realizations per parameter
value, about $10^5$ per value in total (Table~\ref{tab:sweeps}), all with independent initial conditions; we
checked by comparing the recorded initial conditions that the output files of
different series do not overlap. Repeated initial conditions occur only for
$P_x\le2$, where the number of possible configurations is small, and are kept as
genuine samples. The size of the data set is given in Appendix~\ref{app:study}.

Runs at the reference point produced by different sweeps ($r=1$ in the radius
sweep, $P_x=\mathrm{round}(N/3)$ in the composition sweep, $L=8,9$ in the size
sweep, and $\alpha=3/8$ in the threshold sweep, which is the same rule) are
statistically equivalent and are pooled. Table~\ref{tab:series} supports this:
the fractions of absorbing, frozen and oscillating runs obtained in the four
sweeps agree within their statistical errors ($\chi^2$ test of homogeneity,
$p=\prelim{0.32}$ at $L=8$ and $p=\prelim{0.23}$ at $L=9$). Moreover, the
threshold and radius sweeps, with more than $10^5$ runs each, find independently
exactly the same \prelim{7} tuples at $L=8$ and \prelim{12} at $L=9$, and the
smaller sweeps find a subset of them.

\begin{table}[!htbp]
\caption{Fate of the runs at the reference point in each of the sweeps that
contain it (percentages of the runs of each sweep; for the threshold sweep, the
runs at $\alpha=3/8$). The sweeps are statistically indistinguishable, which
justifies pooling them.}
\label{tab:series}
\footnotesize
\begin{ruledtabular}
\begin{tabular}{llrrrr}
$L$ & Sweep & Runs ($10^3$) & Absorbing & Frozen & Oscillating\\
\colrule
8 & radius $r$      & \prelim{123} & \prelim{78.9} & \prelim{18.0} & \prelim{3.1}\\
  & threshold $\alpha$ & \prelim{111} & \prelim{79.1} & \prelim{17.9} & \prelim{3.0}\\
  & composition $P_x$ & \prelim{110}  & \prelim{79.1} & \prelim{17.8} & \prelim{3.2}\\
  & size $L$        & \prelim{111} & \prelim{78.8} & \prelim{18.0} & \prelim{3.1}\\
\colrule
9 & radius $r$      & \prelim{122} & \prelim{78.0} & \prelim{17.1} & \prelim{4.9}\\
  & threshold $\alpha$ & \prelim{111} & \prelim{77.8} & \prelim{17.2} & \prelim{5.0}\\
  & composition $P_x$ & \prelim{110}  & \prelim{78.0} & \prelim{17.1} & \prelim{4.9}\\
  & size $L$        & \prelim{109} & \prelim{78.2} & \prelim{16.9} & \prelim{4.9}\\
\end{tabular}
\end{ruledtabular}
\end{table}

\Ssection{Statistical treatment and unseen tuples}{sec:S-stats}

\paragraph{Basin estimates.} Because the initial conditions are drawn uniformly
from the configurations with the prescribed composition, each run is a
Bernoulli trial for every tuple, and the fraction $\hat p$ of runs ending in a
tuple estimates the relative size $p$ of its basin with binomial standard error
$\sqrt{\hat p(1-\hat p)/n}$ (the SE column of Table~\ref{tab:basins}). With
$n\approx\prelim{4.5\times10^5}$ at each size the error
is below \prelim{0.08} percentage points for the fixed points and below
\prelim{0.03} for every oscillating tuple. Figure~\ref{fig:discovery}b shows that
the estimate of the oscillating fraction is stable once a few thousand runs have
been accumulated.

\paragraph{Could new tuples appear?} The number of simulations is necessarily
finite, so the catalogue of tuples could miss attractors with very small basins.
This is the classical problem of estimating the unseen part of a population from
a sample, and we address it with three standard tools, which use only the
numbers $f_k$ of tuples observed exactly $k$ times among the $n$ runs.
\begin{enumerate}
\item The \emph{Good--Turing} estimator gives the total probability of all the
tuples not yet observed, i.e.\ the probability that the next run ends in a new
tuple, as $\hat P_0=f_1/n$ \cite{Good1953}. Tuples seen only once are the
signature of a population that has not been sampled to exhaustion.
\item The \emph{Chao1} estimator gives a lower bound for the total number of
tuples, $\hat S_\mathrm{Chao1}=S+f_1^2/(2f_2)$, where $S$ is the number of
tuples observed (with $f_1(f_1-1)/2$ in place of $f_1^2/(2f_2)$ when $f_2=0$)
\cite{Chao1984}.
\item The \emph{Good--Toulmin} estimator gives the expected number of new
tuples in $m$ further runs,
\[\hat U(m)=\sum_{k\ge1}(-1)^{k+1}\left(\frac{m}{n}\right)^k f_k,\]
which is reliable for $m\le n$ \cite{Good1956}.
\end{enumerate}
In addition, an individual tuple with basin $p$ is observed at least once with
probability $1-(1-p)^n$, so that any tuple with $p\ge3/n$ would have been seen
with probability of at least 95\%.

Table~\ref{tab:Lsweep} gives the result for every size. At $L=8$ and $L=9$ no
tuple has been observed only once: the rarest ones, $(6,12,4,6,2)$ at $L=8$
and $(21,63,63,21,7)$ at $L=9$, have \prelim{11} and \prelim{19} runs. Hence
$\hat P_0=0$, $\hat S_\mathrm{Chao1}=S$ (\prelim{7} and \prelim{12} tuples), and
$\hat U=0$ for the next $n/10$ runs; any tuple missing from the catalogue has a
basin smaller than about \prelim{$7\times10^{-6}$} of the configuration space,
and cannot change the frequencies of Table~\ref{tab:basins} beyond their
statistical error. The same holds for the values of $\Tpop$ alone: no new
population period is expected at $L=8$ (values 1, 6, 22) or $L=9$ (values 1, 14,
18, 21, 23, 24). The discovery curves of Fig.~\ref{fig:discovery}a show how this
saturation was reached: with $10^4$ runs, about \prelim{5} of the \prelim{7}
tuples of $L=8$ and \prelim{10} of the \prelim{12} of $L=9$ had appeared, and the
complete list required a median of \prelim{$3.2\times10^4$} runs at both sizes over random orderings of the data, about a
tenth of the runs available. At $L=5$--$7$, with more than $10^5$ runs
each, only one or two singletons remain at $L=6$ and $L=7$
($\hat P_0\approx10^{-5}$); the rare color-rotating cycles of $L=7$ in
Table~\ref{tab:catalogue} were all found in the last series.

The $L$ sweep has at least $10^5$ runs per size up to $L=16$, and there the
estimates are also small: $f_1=0$ at $L=10$ and 11, and $\hat P_0\approx1\times10^{-5}$,
$2\times10^{-5}$, $3\times10^{-5}$ and $8\times10^{-6}$ at $L=12$, 13, 14 and 15
(\prelim{$1.2\times10^5$} runs at $L=15$). At $L=16$, with the same number of runs, seven tuples
are still singletons and $\hat P_0\approx6\times10^{-5}$. The situation is different for $L\ge17$, with
\prelim{$2.2\times10^4$} runs per size (\prelim{$5.1\times10^4$} for $L=17$). There singletons are common, and the Good--Turing
estimate is that \prelim{0.01--0.07\%} of further runs would end in a tuple not
yet observed, growing with $L$; the Chao1 estimates exceed the observed number of
tuples by up to a factor of \prelim{1.7} ($L=24$). For $\Tpop$ alone the
estimates are similar. The frequencies of the main classes in
Table~\ref{tab:Lsweep} are reliable, but the catalogue of tuples at $L\ge16$ is
not complete, and the rare long periods listed in Sec.~\ref{sec:S-further} are a
sample of a larger set.

\begin{figure}[!htbp]
\centering
\includegraphics[width=\columnwidth]{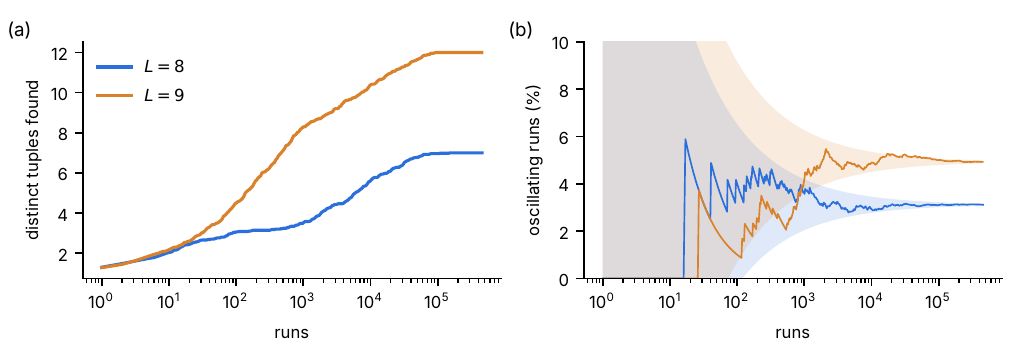}
\caption{Convergence of the sampling at the reference point ($\alpha=1/3$ or
$3/8$, $r=1$, $P_x=\mathrm{round}(N/3)$; all runs pooled).
(a) Number of distinct tuples found as a function of the number of runs, averaged
over 50 random orderings of the runs. (b) Running estimate of the percentage of
oscillating runs for one random ordering; the bands are the 95\% binomial
intervals around the final value.}
\label{fig:discovery}
\end{figure}

\begin{table*}[!htbp]
\caption{Main simulations at the reference point as a function of $L$ ($\alpha=1/3$
or $3/8$, $r=1$, $P_x=\mathrm{round}(N/3)$; all runs pooled).
Fixed: absorbing plus frozen runs; wave: traveling waves; $\Tpop=24$, $\Tpop=80$:
runs with that population period; $\tau_{1/2}$: median transient. The last
columns give the sample-coverage statistics of Sec.~\ref{sec:S-stats}: $S$ is
the number of distinct tuples observed (including the trivial tuple of the fixed
points, counted separately for absorbing and frozen runs), $f_1$ and $f_2$ the
numbers of tuples observed exactly once and twice, $\hat P_0=f_1/n$ the
Good--Turing estimate of the probability that a new run ends in an unobserved
tuple, and $\hat S_\mathrm{Chao1}$ the Chao1 estimate of the total number of
tuples (``--'' when $f_1=0$, in which case $\hat S_\mathrm{Chao1}=S$). The last
three columns, $S_\mathrm{pop}$, $f_{1,\mathrm{pop}}$ and
$\hat S_\mathrm{Chao1,pop}$, are the same statistics computed for the values of
$\Tpop$ alone instead of the complete tuples.}
\label{tab:Lsweep}
\footnotesize\setlength{\tabcolsep}{3.5pt}
\begin{ruledtabular}
\begin{tabular}{rrrrrrrrrrrrrrr}
 & & \multicolumn{4}{c}{\% of runs} & & \multicolumn{5}{c}{complete tuples} & \multicolumn{3}{c}{values of $\Tpop$ only}\\
\cline{3-6}\cline{8-12}\cline{13-15}
$L$ & $n$ ($10^3$) & fixed & wave & $\Tpop{=}24$ & $\Tpop{=}80$ & $\tau_{1/2}$ & $S$ & $f_1$ & $f_2$ & $\hat P_0$ (\%) & $\hat S_\mathrm{Chao1}$ & $S_\mathrm{pop}$ & $f_{1,\mathrm{pop}}$ & $\hat S_\mathrm{Chao1,pop}$\\
\colrule
5 & 129 & 97.9 & 2.1 & 0.0 & 0.0 & 5 & 3 & 0 & 0 & 0 & -- & 1 & 0 & --\\
6 & 129 & 96.8 & 3.2 & 0.0 & 0.0 & 8 & 8 & 1 & 0 & 0.001 & 8 & 5 & 1 & 5\\
7 & 121 & 96.9 & 3.1 & 0.0 & 0.0 & 11 & 8 & 2 & 2 & 0.002 & 9 & 6 & 2 & 7\\
8 & 456 & 96.9 & 3.0 & 0.0 & 0.0 & 14 & 7 & 0 & 0 & 0 & -- & 3 & 0 & --\\
9 & 452 & 95.1 & 3.1 & 1.0 & 0.0 & 19 & 12 & 0 & 0 & 0 & -- & 6 & 0 & --\\
10 & 109 & 94.1 & 3.1 & 0.1 & 0.0 & 25 & 14 & 0 & 1 & 0 & -- & 12 & 0 & --\\
11 & 109 & 88.3 & 3.3 & 0.3 & 0.0 & 33 & 13 & 0 & 2 & 0 & -- & 8 & 0 & --\\
12 & 109 & 90.6 & 3.3 & 2.2 & 0.0 & 44 & 19 & 1 & 2 & 0.001 & 19 & 14 & 0 & --\\
13 & 109 & 86.9 & 4.0 & 1.4 & 0.0 & 62 & 19 & 2 & 2 & 0.002 & 20 & 11 & 2 & 12\\
14 & 109 & 81.6 & 5.2 & 2.6 & 0.0 & 78 & 37 & 3 & 5 & 0.003 & 38 & 25 & 2 & 26\\
15 & 122 & 73.6 & 8.1 & 8.7 & 0.0 & 96 & 25 & 1 & 2 & 0.001 & 25 & 13 & 1 & 13\\
16 & 122 & 69.8 & 8.8 & 12.1 & 0.0 & 116 & 45 & 7 & 1 & 0.006 & 70 & 31 & 4 & 35\\
17 & 51 & 63.8 & 9.4 & 15.8 & 0.0 & 135 & 35 & 7 & 2 & 0.01 & 47 & 27 & 6 & 36\\
18 & 22 & 63.8 & 12.5 & 16.0 & 0.0 & 151 & 33 & 6 & 3 & 0.03 & 39 & 22 & 5 & 28\\
19 & 22 & 57.2 & 15.7 & 15.8 & 4.0 & 168 & 30 & 7 & 2 & 0.03 & 42 & 21 & 5 & 27\\
20 & 22 & 49.0 & 18.5 & 16.6 & 7.8 & 178 & 36 & 9 & 2 & 0.04 & 56 & 27 & 9 & 68\\
21 & 22 & 46.5 & 19.8 & 16.4 & 11.1 & 195 & 42 & 9 & 4 & 0.04 & 52 & 30 & 9 & 50\\
22 & 22 & 40.1 & 18.4 & 19.8 & 14.7 & 206 & 36 & 4 & 5 & 0.02 & 38 & 25 & 4 & 28\\
23 & 22 & 35.3 & 16.5 & 23.7 & 18.8 & 215 & 38 & 6 & 4 & 0.03 & 42 & 26 & 4 & 28\\
24 & 22 & 31.1 & 15.2 & 24.0 & 21.9 & 224 & 42 & 11 & 2 & 0.05 & 72 & 25 & 5 & 31\\
25 & 22 & 27.7 & 14.8 & 26.4 & 24.4 & 239 & 46 & 12 & 6 & 0.05 & 58 & 33 & 8 & 39\\
26 & 22 & 26.0 & 13.0 & 23.6 & 27.1 & 247 & 39 & 9 & 2 & 0.04 & 59 & 30 & 8 & 46\\
27 & 22 & 22.6 & 11.9 & 29.9 & 28.8 & 265 & 44 & 11 & 4 & 0.05 & 59 & 31 & 8 & 39\\
28 & 22 & 20.2 & 10.4 & 31.2 & 30.0 & 282 & 52 & 15 & 7 & 0.07 & 68 & 36 & 12 & 72\\
29 & 22 & 17.3 & 10.1 & 30.4 & 31.0 & 296 & 46 & 11 & 6 & 0.05 & 56 & 33 & 8 & 41\\
30 & 22 & 16.1 & 8.9 & 35.4 & 31.8 & 319 & 57 & 16 & 7 & 0.07 & 75 & 38 & 11 & 58\\
\end{tabular}
\end{ruledtabular}
\end{table*}

\Ssection{Further results}{sec:S-further}

\paragraph{Aggregated versus microscopic periods.} The population period is the
only one that can be measured from census data, so it is natural to ask how much
of the tuple it determines. Pooling all sizes and radii at the reference
threshold, the oscillating runs fall into \prelim{384} distinct tuples. Among
them there is no pair that differs only in $\Tpop$: whenever $\Tst$, $\Tcol$,
$\Tgeo$ and $\TG$ coincide, $\Tpop$ coincides too, so that in our data $\Tpop$
is a function of the microscopic periods. The same holds for $\Tgeo$. The
converse is far from true. A single value of $\Tpop$ can hide many different
attractors: $\Tpop=1$ is shared by \prelim{38} tuples (the waves of all sizes,
with and without equal bands, and the rotated cycle $(1,16,16,4,4)$),
$\Tpop=24$ by \prelim{45} (the stationary and drifting cycles of Sec.~\ref{sec:t24}),
and $\Tpop=45$, 96, 99 and 120 by \prelim{23}, \prelim{20}, \prelim{16} and
\prelim{21}. Even
within one size the census misses structure: at $L=9$, $\Tpop=18$ corresponds
to three different attractors, with $\Tst=18$, 54 and 108, and $\Tpop=21$ to two,
with $\Tst=63$ and 189 (Table~\ref{tab:basins}). Pairs of tuples that differ in
a single period exist only for $\Tcol$ (e.g.\ waves with and without equal bands),
for $\Tst$ (equal-band waves at sizes $L$ and $3L$) and for $\TG$. The last case is
the most instructive: for $\Tpop=24$ and $L=12$, 15, 18, \dots, 30 the
drifting cycles occur with $\TG=8$ and with $\TG=24$ while all the other periods
coincide, e.g.\ $(24,96,96,24,8)$ and $(24,96,96,24,24)$ at $L=12$; only $\TG$
reveals that the first one has a hidden color--translation symmetry.

\paragraph{The family with $\Tpop=24$.} The population period 24 appears at
every size from $L=9$ to $L=30$ (Table~\ref{tab:Lsweep}) and up to $L=40$
(Sec.~\ref{sec:tests}); Sec.~\ref{sec:t24}
analyzes it in detail. From $L=15$ on, most of these attractors are stationary,
$(24,24,8,24,8)$, with a minority of the variant $(24,24,24,24,24)$. The state
period of the drifting members follows Eq.~\eqref{eq:t24},
$\Tst=24L/\gcd(L,3)$, in \prelim{98.7\%} of the drifting runs; the exceptions are
the runs at $L=10$, all with $\Tst=120$, a variant at $L=13$ with twice that value
(\prelim{93} runs), a few runs at $L=18$ and $L=24$ with three halves and two thirds
of it, and two runs at $L=15$ with $\Tst=360$. All of them, like the regular
members, satisfy $\Tst=\TG\,m$ with the generators of Table~\ref{tab:t24}; they
differ only in the translation per period. The defect counts and the generators
were obtained by re-simulating up to three distinct final configurations of each
of the \prelim{82} members of the family (pairs of size and tuple) and up to six for the defect statistics. The second large-size family,
$\Tpop=80$, appears from $L=19$ and always with the tuple $(80,80,80,80,80)$, i.e.\
as an asymmetric cycle ($h=\mathrm{id}$); it also contains one pair of opposite
point defects.

\Ssection{Simulations on larger lattices}{sec:S-larger}

To reduce the errors of the $L$ sweep we also ran the independent
implementation with \prelim{2000} realizations per size for $L=5,\dots,30$ at
the reference point, recording the class of the attractor, the transient and
$\Tst$, and a second set of \prelim{500} realizations per size that also records
$\Tpop$ and $\Tcol$. Exact detection is still feasible at $L=30$, where the median
transient in these runs is about \prelim{310} steps, because only one trajectory is stored at a
time. The two codes agree: at $L=30$ the fraction of runs ending in a fixed point
is \prelim{16.1\%} in the main simulations (\prelim{$2.2\times10^4$} runs) and \prelim{15.9\%}
in the independent ones (2000 runs), and at $L=21$ it is \prelim{46.5\%} and
\prelim{45.6\%}, a difference of \prelim{0.8} combined standard errors. Both sets of
simulations are shown in Fig.~\ref{fig:size}.

\end{document}